\documentclass[11pt]{amsart}
\usepackage{amsmath,amssymb,amsthm,amsfonts}
\usepackage{paralist}
\usepackage{booktabs}
\usepackage[usenames,dvipsnames]{color}
\usepackage{comment}
\usepackage{graphicx,epsf}
\usepackage{graphics}
\usepackage{epsfig}
\usepackage{epstopdf}
\usepackage{psfrag}
\usepackage{latexsym}
\usepackage[latin1]{inputenc}
\usepackage[shortlabels]{enumitem}
\usepackage{thmtools}
\usepackage{url}
\usepackage[all,cmtip]{xy}
\usepackage{hyperref} 
\usepackage{pifont}

\usepackage[enableskew]{youngtab}
\usepackage{ytableau}

\title[On the largest Kronecker and LR--coefficients]
{On the largest Kronecker and \\ Littlewood--Richardson
 coefficients }

\author[Igor Pak, Greta Panova and Damir Yeliussizov]{
Igor Pak$^\star$, \ \  Greta Panova$^\dagger$ \ \ and \ \ Damir Yeliussizov$^\star$}

\thanks{\ \today}
\thanks{\thinspace ${\hspace{-.45ex}}^\star$Department of Mathematics,
UCLA, Los Angeles, CA~90095.
\hskip.06cm
Email:
\hskip.06cm
\texttt{\{pak,\ts damir\}@math.ucla.edu}}

\thanks{{\thinspace ${\hspace{-.45ex}}^\dagger$Department of Mathematics,
 UPenn, Philadelphia, PA~19104, and IAS, Princeton, NJ~08540}}
\thanks{\ Email: \hskip.06cm
\texttt{panova@math.upenn.edu}}

\newcommand{\VK}{\psi}

\newcommand{\SYT}{\operatorname{SYT}}

\newcommand{\Conj}{\operatorname{Conj } }

\DeclareMathOperator{\area}{area}

\newtheorem{thm}{Theorem}[section]
\newtheorem{lemma}[thm]{Lemma}

\newtheorem{cor}[thm]{Corollary}
\newtheorem{prop}[thm]{Proposition}
\newtheorem{conj}[thm]{Conjecture}

\newtheorem{rem}[thm]{Remark}

\newtheorem{Def}[thm]{Definition}

\numberwithin{equation}{section} 

\def\wh{\widehat}

\def\sq{\square}

\def\nn{\mathbb N}
\def\bss{\mathbb S}
\def\cc{\mathbb C}
\def\rr{\mathbb R}

\def\Ga{\Gamma}

\def\la{\lambda}

\def\de{\delta}

\def\al{\alpha}
\def\be{\beta}
\def\om{\omega}

\def\ka{\rho}
\def\ve{\varepsilon}

\def\vp{\varphi}

\def\cC{\mathcal C}

\def\cP{\mathcal P}

\def\ssu{\subset}
\def\sss{\supset}
\def\wt{\widetilde}
\def\<{\langle}
\def\>{\rangle}

\def\GL{ {\text {\rm GL} } }

\def\Ups{\Upsilon}

\def\rK{\text{{\rm \textbf{K}}}}
\def\rLR{\text{{\rm \textbf{C}}}}
\def\rA{\text{{\rm \textbf{A}}}}

\def\rD{\text{{\rm \textbf{D}}}}
\def\rC{\text{{\rm \textbf{C}}}}
\def\rF{\text{{\rm \textsf{F}}}}

\def\0{{\mathbf 0}}

\def\hba{\hslash}

\def\.{\hskip.06cm}
\def\ts{\hskip.03cm}

\def\nin{\noindent}

\def\co{c_1}
\def\ct{c_2}

\def\WI{\text{\small {\rm (\textit{i})}}}
\def\WII{\text{\small {\rm (\textit{ii})}}}
\def\WIII{\text{\small {\rm (\textit{iii})}}}

\begin{document}

\begin{abstract}
We give new bounds and asymptotic estimates for
Kronecker and Littlewood--Richardson coefficients.
Notably, we resolve Stanley's questions on the
shape of partitions attaining the \emph{largest} Kronecker
and Littlewood--Richardson coefficients. We apply the
results to asymptotics of the number of standard Young tableaux
of skew shapes.
\end{abstract}


\maketitle

\section{Introduction} \label{sec:intro}

What is largest dimension $\rD(n)$ of an irreducible representation
of~$S_n$?  Which partitions correspond to the largest representations?
These questions are both classical and surprisingly challenging.
They have been heavily studied in both combinatorics and probability,
in connection to the \emph{longest increasing subsequences}.
We refer to~\cite{AD,BDJ,Rom} for the history of this problem.

\smallskip

In algebraic combinatorics and related fields, the \emph{Kronecker}
and \emph{Littlewood--Richardson} $($LR$-)$ \emph{coefficients}
play a crucial role.  They are the structure constants in the
ring of characters of~$S_n$ and $\GL_N(\cc)$, respectively:
$$\chi^\mu \cdot \chi^\nu \, = \, \sum_\nu \. g(\la,\mu,\nu) \. \chi^\la
\quad \ \text{and} \ \quad s_\mu \cdot s_\nu \, = \, \sum_\la \. c^\la_{\mu,\nu} \ts s_\la\ts.
$$
These coefficients have been intensely studied from combinatorial
(see e.g.~\cite{EC2,vL}), geometric (see e.g.~\cite{Ful}), probabilistic
(see~\cite{Bia}), and computational point of view (see e.g.~\cite{Bur,BI}).
Yet, relatively little is known about the asymptotics of these
coefficients (see~$\S$\ref{ss:finrem-hist}).

In recent years, Stanley computed asymptotic of the \emph{largest}
Kronecker and LR--coefficients:

\begin{thm}[Stanley~\cite{EC2-supp,Stanley-kron}] \label{t:stanley-asy}
We have:
$$\aligned
(\ast) \qquad \quad \ts\  \max_{\la\vdash n} \. \max_{\mu\vdash n} \. \max_{\nu\vdash n}
\,\, g(\la,\mu,\nu) \, & = \, \sqrt{n!} \,\. e^{-O(\sqrt{n})}\ts, \\ 
(\ast\ast) \qquad \, \max_{0\le k\le n} \. \max_{\la\vdash n} \. \max_{\mu\vdash k} \.
\max_{\nu\vdash n-k}  \,\, c^\la_{\mu,\nu} \, & = \, 2^{n/2 \ts \ts - \ts O(\sqrt{n})}\ts.
\endaligned
$$
\end{thm}

Having obtained these interesting asymptotics, Stanley asked
in~\cite{EC2-supp,Stanley-kron}:
\ts {\bf {\em What partitions attain these maxima}?}  \ts
His original proof used certain summation identities which gave no
indication on how to answer either question.\footnote{R.~P.~Stanley,
personal communication (2017).}

In this paper we largely resolve questions for both~$(\ast)$ and~$(\ast\ast)$,
and the answers turn out to be the \emph{Vershik--Kerov--Logan--Shepp}
(VKLS) \emph{shapes}.  These are the shapes that attain the largest
dimension~$\rD(n)$ mentioned above.  The precise formulations of both
results is somewhat technical (see next section), with notable differences
between two cases.  Contrary to the many other questions in the area,
the Kronecker coefficients are actually easier to understand than
the LR--coefficients, for reasons that will become apparent only
much later in the paper (see $\S$\ref{ss:finrem-groups}).

Somewhat surprisingly, and in contrast with the technical proofs
in the earlier results, we are able to resolve both questions by
using basic character estimates and utilizing various existing
technical results in the area.
%

Throughout the paper, we worked hard to simplify and streamline the proofs,
and remove $S_n$--specific tools to make sure that our bounds can be
extended to general finite groups.  In the companion paper~\cite{PPY},
we present slightly weaker but much more general bounds for
\emph{induced multiplicities}, which extend both the Kronecker
and the LR--coefficients (cf.\ $\S$\ref{ss:finrem-groups}).

\smallskip

\subsection{Kronecker coefficients}
Recall that the largest dimension \ts $\rD(n) := \max_{\la \vdash n} \ts f^\la$ \ts
satisfies
$$
\sqrt{n!} \, e^{-c_1\ts \sqrt{n}\ts (1+o(1))} \. \le \. \rD(n) \. \le \.  \sqrt{n!} \, e^{-c_2\ts \sqrt{n}\ts (1+o(1))}
$$
for some $c_1>c_2>0$~\cite{VK2}.  This motivates the following definition.

\smallskip

\begin{Def}
A sequence of partitions \ts $\{\la^{(n)}\vdash n\}$ \ts is called
\emph{Plancherel} if for some \ts $a>0$, we have:
$$
f^{\la^{(n)}} \, \ge \, \sqrt{n!} \, e^{-a \ts \sqrt{n}}
\quad \text{for all} \ \, n\ts.
$$
\end{Def}

\smallskip

\begin{thm}[{\cite{VK} and~\cite{LS}}] \label{t:VKLS}
Every Plancherel partition sequence 
has {\rm VKLS} shape.
\end{thm}

\smallskip

\nin
See~$\S$\ref{ss:VKLS} for precise definitions of VKLS shapes.
We write \ts $f^{\la^{(n)}} = \sqrt{n!} \, e^{-O(\sqrt{n})}$ \ts when
we do not want to specify the constant implied by the $O(\cdot)$ notation.

\smallskip

\begin{thm} \label{t:kron-intro}
Let $\{\la^{(n)}\vdash n\}$, $\{\mu^{(n)}\vdash n\}$, $\{\nu^{(n)}\vdash n\}$ be
three partitions sequences, such that
\begin{equation}
\label{eq:Kron-intro}
g\bigl(\ts \la^{(n)},\ts\mu^{(n)},\ts\nu^{(n)}\ts\bigr) \, = \, \sqrt{n!} \, e^{-O(\sqrt{n})}\..
\end{equation}
Then all three partition sequences are Plancherel.  Conversely, for every
two Plancherel partition sequences \ts $\{\la^{(n)}\vdash n\}$ \ts and \ts
$\{\mu^{(n)}\vdash n\}$, there exists a Plancherel partition sequence \ts
$\{\nu^{(n)}\vdash n\}$, s.t.~\eqref{eq:Kron-intro} holds.
\end{thm}

This resolves Stanley's question on Kronecker coefficients as it implies
that $\rK(n)$ must be attained at VKLS shapes. We also obtain the several
extensions and refinements of these results.

\smallskip

\subsection{Littlewood--Richardson coefficients} Define

$$
\rC(n,k) \, := \, \max_{\la\vdash n} \. \max_{\mu \vdash k}\. \max_{\nu \vdash n-k}
 \. c^\la_{\mu,\nu}\..
$$

\smallskip

\begin{thm}\label{t:LR-intro-asy}
For all $n> k\ge 1$, we have:
$$\sqrt{\binom{n}{k}} \. e^{-d \ts \sqrt{n}} \. \le  \. \rC(n,k) \. \le  \. \sqrt{\binom{n}{k}}
$$
for some universal constant \ts $d>0$.
\end{thm}

This refines Stanley's asymptotic bound~$(\ast\ast)$ in Theorem~\ref{t:stanley-asy}.
Note that $\binom{n}{k}$ is the dominant term of the asymptotics for all \ts
$\sqrt{n} \le k \le n-\sqrt{n}$.  Note also that $\{\rC(n,k)\}$  are not
unimodal for some fixed~$n\ge 10$, see~$\S$\ref{a:LR} (see also~$\S$\ref{ss:LR-cont}).

\begin{thm} \label{t:LR-intro1}
Fix \ts $0<\theta<1$ \ts and let \ts $k_n := \lfloor\theta n\rfloor$.
Then:
\smallskip

\nin
$\WI$ \ts for every Plancherel partition sequence $\{\la^{(n)}\vdash n\}$,
there exist Plancherel partition sequences $\{\mu^{(n)}\vdash k_n\}$
and $\{\nu^{(n)}\vdash n-k_n\}$, s.t.
%
$$
(\divideontimes) \qquad \ c^{\la^{(n)}}_{\mu^{(n)},\. \nu^{(n)}} \, = \, \binom{n}{k_n}^{1/2} \, e^{-O(\sqrt{n})}\.,
$$
%

\smallskip

\nin
$\WII$ \ts  for all Plancherel partition sequences
$\{\mu^{(n)}\vdash k_n\}$ and $\{\nu^{(n)}\vdash n-k_n\}$, there exists a
Plancherel partition sequence $\{\la^{(n)}\vdash n\}$, s.t.~$(\divideontimes)$ holds,
\smallskip

\nin
$\WIII$ \ts  for all Plancherel partition sequences
$\{\la^{(n)}\vdash n\}$ and $\{\mu^{(n)}\vdash k_n\}$, there exists a
partition sequence $\{\nu^{(n)}\vdash n-k_n\}$, s.t.
$$
f^{\nu^{(n)}} \, = \, \sqrt{n!} \, e^{-O(n^{2/3}\log n)} \quad \text{and} \quad
c^{\la^{(n)}}_{\mu^{(n)},\. \nu^{(n)}} \, = \, \binom{n}{k_n}^{1/2} \, e^{-O(n^{2/3}\log n)}\..
$$
\end{thm}

Combined, these results resolve Stanley's question on LR--coefficients.
Note that Proposition~\ref{p:non-robust} shows that one cannot conclude
that LR--coefficients are large when the dimensions of all three
representations are large.  Note also that part~$\WIII$ requires
the bound to hold in a wider setting than~$\WI$.  It is thus
unsurprising that it has weaker conclusions.

\smallskip

We also obtain a partial converse:

\begin{thm} \label{t:LR-intro2}
Fix \ts $\theta=1/2$ \ts and let \ts $k_n := \lfloor\theta n\rfloor$.
Suppose $\{\la^{(n)}\vdash n\}$, $\{\mu^{(n)}\vdash k_n\}$ and $\{\nu^{(n)}\vdash n-k_n\}$ be
three partitions sequences, s.t.
\begin{equation} \label{eq:LR-intro-weak}
c^{\la^{(n)}}_{\mu^{(n)},\. \nu^{(n)}} \, = \, \binom{n}{n/2}^{1/2} \, e^{-O(n/\log n)}\..
\end{equation}
Then all three partition sequences satisfy:
$$f^{\la^{(n)}} \, = \, \sqrt{n!} \, e^{-O(n)}, \quad
f^{\mu^{(n)}} \, = \, \sqrt{(n/2)!} \, e^{-O(n)} \quad  \text{and} \quad
f^{\nu^{(n)}} \, = \, \sqrt{(n/2)!} \, e^{-O(n)}\..
$$
\end{thm}


\smallskip

\subsection{Structure of the paper} We start with definitions and
known results in Section~\ref{s:def}.  Kronecker coefficients and
the proof of Theorem~\ref{t:kron-intro} are given in Section~\ref{s:Kron}.
We continue with a lengthy Section~\ref{s:LR} on LR--coefficients,
including the proof of Theorem~\ref{t:LR-intro-asy} and parts~$\WI$, $\WII$
of Theorem~\ref{t:LR-intro1}.  We continue with a short but technical
Section~\ref{s:LR-imply-dim}, again on LR--coefficients, where we prove
Theorem~\ref{t:LR-intro2}.  We apply our results in Section~\ref{s:skew}
to the estimate the number of standard Young tableaux of skew shape,
proving the remaining part~$\WII$ of Theorem~\ref{t:LR-intro1}.
We conclude with final remarks and open problems (Section~\ref{s:fin-rem}).

\bigskip

\section{Definitions, notation and background results}\label{s:def}

\subsection{Partitions}\label{ss:def-part}
Let \ts $p(n)= \#\{ \la \vdash n\}$ \ts be the \emph{number of partitions} of~$n$.
The \emph{Hardy--Ramanujan asymptotic formula} gives:
$$p(n) \. \sim  \. \frac{1}{4\ts n \ts \sqrt{3}}\.
e^{\pi\ts \sqrt{\frac{2n}{3}}} \quad \text{as} \ \ n\to \infty.
$$
We also need the log-concavity of the partition function:
$$
p(n-k) \cdot p(n+k) \, \ge \, p(n)^2,
$$
for all $n>25$, $k=1$, and for all $n>k>1$~\cite{DP}.

\smallskip

\subsection{Dimensions and Young tableaux}\label{ss:def-yt}
We assume the reader is familiar with standard definition and
notation in the area; we follow~\cite{EC2} in most cases.

Let \ts $f^{\la/\mu}=\#\SYT(\la/\mu)$ \ts denote the number of
\emph{standard Young tableaux} of skew shape $\la/\mu$.
Recall the \emph{hook-length formula} (see e.g.~\cite[$\S$7.21]{EC2}):
\begin{equation}\label{eq:HLF}
f^{\la} \, = \, n!\, \prod_{(i,j)\in \la} \. \frac{1}{h_{ij}}\,,
\end{equation}
where $n=|\la|$ and $h_{ij}= \la_i+\la_j' -i-j+1$ is the
\emph{hook length}.  For skew shapes, we have the following
corollary of the \emph{Naruse hook-length formula}:
\begin{equation}\label{eq:NHLF-lower}
f^{\la/\mu} \, \ge \, n!\, \prod_{(i,j)\in \la/\mu} \. \frac{1}{h_{ij}}\,,
\end{equation}
where $n=|\la/\mu|$, see~\cite{MPP}.

Let $s_\la(x_1,x_2,\ldots)$ denote the \emph{Schur function}.
For the evaluations \ts $s_\la(1^m):= s_\la(1,\ldots,1)$, $m$ ones,
we have the following \emph{hook-content formula}:
\begin{equation}\label{eq:HCF}
s_{\lambda}(1^{m}) \, = \, \prod_{(i,j)\in \la} \. \frac{m +j-i}{h_{ij}}\..
\end{equation}
(see e.g.~\cite[$\S$7.21]{EC2}).

\smallskip

\subsection{Largest dimension} \label{ss:def-largest-dim}
Vershik and Kerov~\cite{VK2} proved that for all $n$ large enough:
\begin{equation}\label{eq:VK}
\sqrt{n!} \. e^{-\co\ts \sqrt{n}\ts (1+o(1))} \, \le \, \rD(n) \, \le \,  \sqrt{n!} \. e^{-\ct\ts \sqrt{n}\ts (1+o(1))} ,
\end{equation}
where
\begin{equation}\label{eq:VK-const}
\co \. = \. \pi\ts \sqrt{\frac{1}{6}} \. \approx \ts 1.2825 \quad \text{and} \quad \ct \. = \. \frac{\pi-2}{\pi^2} \. \approx \ts 0.1157\..
\end{equation}
See~\cite[\href{https://oeis.org/A003040}{A003040}]{OEIS} and $\S$\ref{a:LR}
for the initial values and~\cite{VP} for numerical experiments.
Note that the lower bound follows from the Burnside identity:
\begin{equation}\label{eq:Burnside}
\sum_{\la\vdash n} \. \bigl(f^\la\bigr)^2  \, = \, n!
\end{equation}
However, the upper bound is rather remarkable.
In fact, Vershik and Kerov proved in~\cite{VK2} that with respect to the
Plancherel measure \ts $(f^\la)^2/n!$, almost all dimensions $f^\la$
lie in the interval in~\eqref{eq:VK}.  Moreover,
in the same setting, Bufetov proved~\cite{Buf}:
$$
f^\la \, = \, \sqrt{n!} \. e^{-h\ts \sqrt{n}\ts (1+o(1))} \ \ \. 
\text{a.s.\ts, \, for some} \ \. \co \ge h\ge \ct\ts.
$$
This should be compared with the \emph{Vershik--Kerov--Pass conjecture}:
$$\rD(n) \, = \, \sqrt{n!} \. e^{-a\ts \sqrt{n}\ts (1+o(1))}
$$
for some \. $\ct \le a \le \co$. This was implicitly suggested by Vershik in~\cite{Ver}
(in greater generality), and formally stated by Kerov and Pass~\cite{KP}.  


\smallskip

\subsection{Stable limit shapes} \label{ss:def-limit}
The following definition follows~\cite{MPP}.  Let $\om: [0,a] \to [0,b]$ be a
non-increasing continuous function.  Suppose sequence of partitions $\{\la^{(n)}\}$
satisfies the following property
$$
\bigl(\sqrt{n}-L\bigr)\ts\om \, < \,  \bigl[\la^{(n)}\bigr]\, < \,  \bigl(\sqrt{n}+L\bigr)\ts\om\ts,
\ \quad \text{for some} \ \ L>0\ts,
$$
where we write $[\lambda]$ to denote a function giving the boundary of
Young diagram~$\lambda$.
In this case we say that a sequence of partitions $\{\la^{(n)}\}$
has a \emph{strongly stable shape $\om$}, and write \ts $\la^{(n)} \mapsto \om$.

Suppose we are given two stable shapes $\om,\pi: [0,a] \to [0,b]$,
such that $\pi(x) \le \om(x)$ for all $x\ge 0$.  To simplify the notation,
denote by $\cC=\cC(\om/\pi) \ssu \rr_+^2$ the region between the curves.
One can view $\cC$ as the stable shape of skew shapes, and denote
by $\area(\om/\pi)$ the area of~$\cC$.

Define the \emph{hook function} $\hba: \cC\to \rr_+$
to be the scaled function of the hooks:  \ts
$\hba(x,y):= h\bigl(\lfloor x\sqrt{n}\rfloor,\lfloor y\sqrt{n}\rfloor\bigr)$.

\begin{thm}[{\cite{MPT} and~\cite[$\S$6]{MPP}}] \label{t:stable}
In notation above, suppose \ts $\la^{(n)} \mapsto \om$,
$\mu^{(n)} \mapsto \pi$ and $\area(\om/\pi)=1$. Then we have:
$$
f^{\la^{(n)}/\mu^{(n)}} \, = \, \sqrt{n!} \. e^{c(\om/\pi) \ts n \ts + \ts o(n)}\ts,
$$
for some constant $c(\om/\pi)$ satisfying
$$\Ups(\om/\pi) \. - \. \frac12 \, \le \, c(\om/\pi) \, \le \, \Ups(\om/\pi) \. - \. \frac12 \. + \. \log 2\ts,
$$
where
$$
\Ups(\om/\pi) \. := \. -\. \iint_\cC \. \log \hba(x,y) \, dx \. dy\ts.
$$
\end{thm}

The constant $\Ups(\om/\pi)$ in the theorem is called the \emph{hook integral}
of the limit shape.  The lower bound on $c(\om/\pi)$ in the theorem
follows easily from~\eqref{eq:NHLF-lower}, but the upper bound
requires full power of the Naruse hook-length formula~\cite{MPP}.

\smallskip

\subsection{VKLS shapes} \label{ss:VKLS}
A partition sequence \ts $\{\la^{(n)}\}$ \ts is said to have
\emph{Vershik--Kerov--Logan--Shepp} (VKLS) \emph{shape} \ts if
$$
\left|\ts \frac{1}{\sqrt{n}} \. \bigl[\la^{(n)}\bigr] \. - \. \VK \ts \right|\, < \, C\ts n^{1/6}\quad
\text{for some} \ \ C>0,
$$
where \ts $\VK: \ts [0,2]\to [0,2]$ \ts is obtained by $135^\circ$ rotation
of the curve \ts $\bigl(x,\vp(x)\bigr)$, defined by
\begin{equation}\label{eq:vkls-shape}
\vp(x) \. := \. \frac{2}{\pi} \.\left(\ts x \ts \arcsin \frac{x}{\sqrt{2}} \. + \. \sqrt{2-x^2}\ts \right),
\ \quad
-\ts \sqrt{2}\, \le \. x \. \le \. \sqrt{2}\..
\end{equation}
Here the distance $|\cdot|$ is the \emph{Fr\'echet distance} between the curves.

It follows from Theorem~\ref{t:VKLS}, that $\VK$ is the unique curve of area~1,
s.t. $c(\VK)=0$ and $\Ups(\VK)=1/2$.
The latter holds since the lower bound in~\eqref{eq:NHLF-lower} is an equality for
straight shapes by the hook-length formula~\eqref{eq:HLF}.  This implies:

\begin{prop}  Suppose a partition sequence \ts $\{\la^{(n)}\}$ \ts has a 
strongly stable shape~$\om$, and satisfies
$$f^{\la^{(n)}} \, = \, \sqrt{n!} \. e^{-o(n)}\ts. 
$$
Then \ts $\om=\vp$ \ts defined in~\eqref{eq:vkls-shape}. 
\end{prop}

We refer to~\cite{Rom} for the detailed proof of the VKLS theorem 
(Theorem~\ref{t:VKLS}), and its many interesting applications.

\medskip

\section{Kronecker coefficients} \label{s:Kron}

\subsection{General inequalities} \label{ss:Kron-gen}
The Kronecker coefficient \ts $g(\la,\mu,\nu)$ \ts defined in the introduction
satisfy:
$$
 g(\la,\mu,\nu) \. = \. \bigl\<\chi^\la,\ts \chi^\mu \ts \chi^\nu\bigr\> \. = \. \bigl\<\chi^\la\ts\chi^\mu \ts \chi^\nu, 1\bigr\>.
$$
This implies the symmetries
\begin{equation}\label{eq:Kron-sym-groups}
g(\la,\mu,\nu) \. = \. g(\la,\nu,\mu) \. = \. g(\mu,\la,\nu)  \. = \. \ldots
\end{equation}
In particular, we have a general upper bound:
\begin{equation}\label{eq:Kron-upper-groups}
g(\la,\mu,\nu) \. \le \. f^\la \cdot \min\bigl\{f^\mu/f^\nu, \ts f^\nu/f^\mu\bigr\} \. \le \. f^\la\ts.
\end{equation}

\smallskip

\begin{prop}  \label{p:Kron-main-groups}
Let \ts $\la,\mu,\nu\vdash n$. Suppose \ts $g(\la,\mu,\nu) \ge \rD(n)/a$, for some $a\ge 1$.
Then: \ts $f^\la,  f^\mu, f^\nu \ge \rD(n)/a$.
\end{prop}

\begin{proof}  This follows immediately from~\eqref{eq:Kron-upper-groups}
and the symmetries~\eqref{eq:Kron-sym-groups}.
\end{proof}

\medskip

\subsection{Largest Kronecker coefficients}
Define the \emph{largest Kronecker coefficient} of~$G$ as follows:
$$\rK(n) \. := \. \max_{\la,\mu,\nu \vdash n}  \. g(\la,\mu,\nu)\ts.
$$
Similarly, define
$$
\rA(n) \. := \. \sum_{\la,\mu,\nu \vdash n} \. g(\la,\mu,\nu)^2
$$
\smallskip

\begin{lemma}\label{l:kron-sum-squares}
We have:
\begin{equation}\label{eq:Kron-squares-groups}
\rA(n) \, = \, \sum_{\al \vdash n} \. z_\al\ts,
\end{equation}
where $z_\al = |C(\al)|$ is the size of the centralizer of an element $x\in \al$.
\end{lemma}

\begin{proof}[Proof of Lemma~\ref{l:kron-sum-squares}]
We have:
$$g(\la,\mu,\nu) \, = \, \frac{1}{n!} \. \sum_{x \in G} \. \chi^\la(w)\ts \chi^\mu(x)\ts\chi^\nu(x)\ts.
$$
Hence, we can write the sum of squares as
$$\begin{aligned}
& \sum_{\la,\mu,\nu \vdash n} \. g(\la,\mu,\nu)^2 \, = \, \frac{1}{n!^2} \. \sum_{x,y \in S_n} \. \sum_{\la \vdash n} \. \chi^\la(x)\ts\chi^\la(y) \. \sum_{\mu \vdash n} \. \chi^\mu(x)\ts\chi^\mu(y) \. \sum_{\nu \vdash n} \. \chi^\nu(x)\ts \chi^\nu(y)\\
& \quad = \, \frac{1}{n!^2} \. \sum_{x,y\in S_n} \left(\sum_{\la\vdash n} \. \chi^\la(x)\ts \chi^\la(y)\right)^3 \, = \, \frac{1}{n!^2} \. \sum_{\al \vdash n} \. \left(\frac{n!}{z_\al}\right)^2 \ts (z_\al)^3\, = \, \sum_{\al \vdash n} \. z_\al\..
\end{aligned}
$$
Here the last equality follows from the orthogonality of characters.
\end{proof}

\begin{prop}  \label{p:Kon-max-groups}
We have:
$$
\frac{\sqrt{n!}}{p(n)^{3/2}} \, \le \, \rK(n) \, \le \, \rD(n)\ts.
$$
\end{prop}

\begin{proof}
In~\eqref{eq:Kron-squares-groups}, we have \ts $\text{RHS} \ge z_1=n!$.
This gives the lower bound.  
The upper bound follows from~\eqref{eq:Kron-upper-groups}.
\end{proof}

The following result can be viewed as a converse of
Proposition~\ref{p:Kron-main-groups}.

\begin{thm}  \label{t:Kron-main-groups}
Let \ts $\mu,\ts \nu\vdash n$. Suppose \ts
$f^\mu, \ts f^\nu \ge \rD(n)/a$ \ts
for some \ts $a\ge 1$. Then there exist \ts
$\la\vdash n$, s.t.\
$$
f^\la \, \ge \, \frac{\rD(n)}{a\. \sqrt{p(n)}}\.
\quad \text{and} \quad
g(\la,\mu,\nu) \, \ge \, \frac{\rD(n)}{a^2\ts p(n)}\..
$$
\end{thm}

\begin{proof}
Let $\la\vdash n$ be the partition of the largest term in the RHS of
$$
\frac{\rD(n)^2}{a^2} \, \le \, f^\mu \cdot f^\nu \. = \. \sum_{\la\vdash n} \. g(\la,\mu,\nu)\. f^\la\ts.
$$
On the one hand,
$$
g(\la,\mu,\nu) \, \ge \, \frac{1}{p(n)\cdot \rD(n)} \ts\cdot \ts \frac{\rD(n)^2}{a^2}\, = \, \frac{\rD(n)}{a^2\ts p(n)}\..
$$
On the other hand,
$$
(f^\la)^2 \, \ge \, g(\la,\mu,\nu)\. f^\la \, \ge \, \frac{1}{p(n)} \ts\cdot \ts\frac{\rD(n)^2}{a^{2}}\.,
$$
which implies the result.
\end{proof}

\medskip

\subsection{Refined Kronecker coefficients}
Fix $\la,\mu\vdash n$.  Define the \emph{largest refined Kronecker coefficient}
$$
\rK(\la,\mu) \. := \. \max_{\nu \vdash n}  \. g(\la,\mu,\nu)\.
$$
Clearly, $\rK(\la,\mu) \le \rK(n)$.

\begin{prop}\label{p:Kron-refined-groups}  For all $\la,\mu\vdash n$, we have:
$$
\frac{f^\la \. f^\mu}{\sqrt{p(n)\. n!}} \, \le \,\rK(\la,\mu) \, \le \, \min\bigl\{f^\la,\ts f^\mu\bigr\}.
$$
\end{prop}

\begin{proof} Let
$$
A(\la,\mu) \, := \, \sum_{\nu\vdash n} \. g(\la,\mu,\nu)^2\ts.
$$
Recall Burnside's identity~\eqref{eq:Burnside} and
$$
\sum_{\nu \vdash n} \. g(\la,\mu,\nu) \. f^\nu \, = \, f^\mu\ts f^\la\ts.
$$
Now apply the Cauchy--Schwarz inequality to vectors \.
$\bigl(f^\nu\bigr), \bigl(g(\la,\mu,\nu)\bigr)\in \rr^{p(n)}$,
both indexed by $\nu \vdash n$.  We obtain:
\begin{equation}\label{eq:A-la-mu}
A(\la,\mu) \, \ge \, \frac{(f^\la)^2\ts (f^\mu)^2}{n!}\..
\end{equation}
Therefore, for the maximal term in the summation $A(\la,\mu)$, we have:
$$
\max_{\nu \vdash n} \ts g(\la,\mu,\nu) \, \ge \,
\frac{f^\la\ts f^\mu}{\sqrt{p(n)\ts n!}}\,.
$$
This gives the lower bound.  The upper bound follows from~\eqref{eq:Kron-upper-groups}.
\end{proof}

\begin{cor}  \label{c:Kon-max-refined-groups}
For all $n\in \nn$, have:
$$
\frac{\rD(n)^2}{\sqrt{p(n)\. n!}}  \, \le \,  \rK(n) \, \le \, \rD(n)\..
$$
\end{cor}

\begin{proof}  In Proposition~\ref{p:Kron-refined-groups}, take
\ts $\la,\mu \vdash n$ \ts such that \ts $f^\la=f^\mu = \rD(n)$.
This gives the lower bound.  The upper bound is given in
Proposition~\ref{p:Kon-max-groups}.
\end{proof}

Note that the lower bound in Corollary~\ref{c:Kon-max-refined-groups}
is slightly stronger than the lower bound in Proposition~\ref{p:Kon-max-groups}.
Now the bounds~\eqref{eq:VK} give us:

\begin{cor} We have:
$$
\sqrt{n!} \. e^{-3\ts\co\ts \sqrt{n}\ts (1+o(1))} \, \le \, \rK(n)
\, \le \,  \sqrt{n!} \. e^{-\ct\ts \sqrt{n}\ts (1+o(1))}\ts,
$$
where $\co,\ct>0$ are given in~\eqref{eq:VK-const}.
\end{cor}
It would be interesting to improve either of these two bounds.

\smallskip

\begin{rem}{\rm
The following asymptotic formula is well known and easy to see~\cite[\href{https://oeis.org/A110143}{A110143}]{OEIS}:
\begin{equation}\label{eq:Kron-squares}
\rA(n) \, = \, \sum_{\al\vdash n} \ts z_\al  \, = \, n! \ts \left(1 \. + \. \frac{2}{n^2} \. + \. O\left(\frac{1}{n^3}\right)\right).
\end{equation}
This shows that the lower bounds~\eqref{eq:A-la-mu} in the proof of
Proposition~\ref{p:Kron-refined-groups}
are also asymptotically sharp.  Indeed, in the notation of the
proof of the proposition, we have:
$$
\rA(n) \, = \, \sum_{\la,\ts \mu,\nu \vdash n} \. g(\la,\mu,\nu)^2 \, = \,
\sum_{\la,\ts \mu\vdash n} \. A(\la,\mu) \, \ge \, \sum_{\la,\ts \mu\vdash n} \. \frac{(f^\la)^2 \ts (f^\mu)^2}{n!} \, = \, n!\.,
$$
which implies that the inequalities~\eqref{eq:A-la-mu} are sharp up to a small (additive) error.
}\end{rem}

\medskip

\subsection{Shape of the largest Kronecker coefficients}

\begin{proof}[Proof of Theorem~\ref{t:kron-intro}]
The theorem is a special case of Theorem~\ref{t:Kron-main-groups}, where \ts
$a=e^{-O(\sqrt{n})}$.  Indeed,
both parts of Theorem~\ref{t:kron-intro} either assume or require
$$g\bigl(\la^{(n)},\mu^{(n)},\nu^{(n)}\bigr), \, f^{\la^{(n)}}, \,
f^{\mu^{(n)}}, \,  f^{\nu^{(n)}} \, \ge \, \sqrt{n!} \.\ts e^{-\al\ts\sqrt{n}}
$$
for some $\al\ge 0$, implied by the $O(\cdot)$ notation in the theorem.
Theorem~\ref{t:Kron-main-groups} makes this dependence precise.
\end{proof}

\smallskip

Note that the proof of Theorem~\ref{t:Kron-main-groups} is simple enough
to imply a quantitative version of Biane's concentration of characters
result~\cite{Bia}.

\begin{prop} \label{p:Kron-weak-Biane}
Let $\al, \be, \ve >0$ such that $\al+\be >c_1$. Let \ts
$\la,\mu\vdash n$ such that
$$
f^\la, \, f^\mu \, > \, \sqrt{n!} \.  e^{-\al \ts \sqrt{n}}\ts.
$$
Then:
$$
\sum_{\nu\in \cP_\be(n)} \. g(\la,\mu,\nu) \ts f^\nu \, > \, (1-\ve)\ts f^{\la}\ts f^{\mu} \quad
\text{for all \ts $n$ \ts large enough,}
$$
where the summation is over the set \ts $\cP_\be(n)$ \ts of $\nu$ such
that
$$
f^\nu \, > \, \sqrt{n!} \.  e^{-\be \ts \sqrt{n}}\ts.
$$
\end{prop}

The proposition is phrased in the form to underscore the similarity
with Thm~1.4.1 in~\cite{Bia}.  Of course, Biane's result is
more general as it applied to partitions with other limits shapes~$\la$,
which all must have \ts $f^\la = \sqrt{n!} \ts e^{O(n)}$, see~\cite{MPP}.

\begin{proof}[Proof of Proposition~\ref{p:Kron-weak-Biane}]
We have:
$$
\sum_{\nu\vdash n} \. g(\la,\mu,\nu) \ts f^\nu \, =  \,f^\la \cdot f^\mu \, > \, n! \.  e^{-2 \ts \al \ts \sqrt{n}}
$$
while
$$
\sum_{\nu\notin \cP_\be(n)} \. g(\la,\mu,\nu) \ts f^\nu \, \le \,
\left(\sqrt{n!} \.  e^{-\be \ts \sqrt{n}}\right)^2 p(n) \, = \,
n! \. e^{-2(\be+c_1)\ts\sqrt{n}(1+o(1))}\ts.
$$
This implies the result. \end{proof}

\medskip

\subsection{Tensor square conjectures}  \label{ss:Kron-saxl}
It was suggested in~\cite{HSTZ}
that for all $n$ large enough, there exists $\la\vdash n$, such that
$g(\la,\la,\nu)>0$ for all $\nu\vdash n$.  In other words,
the tensor square \ts $\chi^\la \times \chi^\la$ \ts contains all
characters~$\chi^n$.  One would want to start with
the Plancherel shape $\la$, for which by by Proposition~\ref{p:Kron-weak-Biane}
the tensor square has large average isotypic components.
Note however, that the Plancherel shape is a large
class of partitions, making them hard to study.  For example, one needs
$\la = \la'$ since otherwise $g(\la,\la,1^n)=0$, and there are other
similar small constraints $\la$ must satisfy.

Saxl conjectured~\cite{PPV} that $\la = (k-1,k-2,\ldots,1) \vdash n=\binom{k}{2}$
for large enough~$k$ would have \ts $g(\la,\la,\nu)>0$ \ts for all $\nu\vdash n$.
This is now known for \emph{most} \ts $\nu\vdash n$~\cite{LuS}, but there are
only very weak lower bounds on  \ts $g(\la,\la,\nu)$.  Even when \ts $\la=\nu$,
the Bessenrodt--Behns lower bound \ts $g(\la,\la,\la)\ge 1$ \ts
is the best bound we have in this case~\cite{BB}.
In fact, the largest known values of Kronecker coefficients
for explicitly given families of partitions are subexponential~\cite{PP}.

This all makes the following conjecture both plausible and out
of reach.  Such a result would provide an ultimate extension in
the converse part of Theorem~\ref{t:kron-intro}.

\begin{conj}\label{conj:kron-sharp}
Let $\{\la^{(n)}\vdash n\}$, $\{\mu^{(n)}\vdash n\}$, $\{\nu^{(n)}\vdash n\}$ be
three Plancherel partitions sequences.  Then:
$$
g\bigl(\la^{(n)},\mu^{(n)},\nu^{(n)}\bigr) \. = \. \sqrt{n!} \. e^{-O(n)}\ts.
$$
\end{conj}

Below we show that it is possible to have three large characters
with a zero Kronecker coefficient.  This shows that assumptions
in the conjecture cannot be significantly weakened.

\medskip

\subsection{Vanishing Kronecker coefficients}\label{ss:Kron-regev}
Let us present examples of partition families of large
dimensions with zero Kronecker coefficients

\begin{thm}[Regev~\cite{Reg}]
Let $\la,\mu,\nu\vdash n$ such that $\ell(\la)>\mu_1\cdot \nu_1$.
Then $g(\la,\mu,\nu)=0$.
\end{thm}

\begin{cor} There is a partition family $\{\la^{(n)}\}$,
s.t.\
$$g\bigl(\la^{(n)},\la^{(n)},\la^{(n)}\bigr) \. = \. 0\quad \
\text{and} \quad f^\la \. = \, \sqrt[3]{n!} \. e^{O(n)}.
$$
\end{cor}

\begin{proof}
Let \ts $\la=(a^2)^{a-1} \vdash n$, where \ts $n=(a-1)a^2$.
Note that \ts $\ell(\la) > (\la_1)^2$.
Then \ts $g(\la,\la,\la)=0$ \ts by Regev's theorem.
The degree follows from the hook-length formula~\eqref{eq:HLF}.
We omit the details.
\end{proof}

\smallskip

\begin{thm}[Dvir~\cite{Dvir}]
Let $\la,\mu,\nu\vdash n$ such that $\ell(\la)>|\mu \cap \nu|$.
Then $g(\la,\mu,\nu)=0$.
\end{thm}

\begin{cor} There are three partition families \ts $\{\la^{(n)}\}$, \ts
$\{\mu^{(n)}\}$ \ts and \ts $\{\nu^{(n)}\}$, s.t.\ \ts
$$g\bigl(\la^{(n)},\mu^{(n)},\nu^{(n)}\bigr) \. = \. 0\., \quad  f^{\la^{(n)}} \.  = \. (n!)^{\Theta(1)},
$$
$$
f^{\mu^{(n)}} \. = \. \sqrt{n!} \. e^{-O(\sqrt{n})} \quad \text{and} \quad
f^{\nu^{(n)}} \.  =  \. \sqrt{n!} \. e^{-O(n)}\..
$$
\end{cor}

\begin{proof}
Let \ts $\mu\vdash n$ be a Plancherel shape, and \ts $\nu=(a^{a})$,
where \ts $n=a^2$.  Let \ts $\la = (b^b1^{n-b^2})$ \ts for \ts
$b=\lfloor\ve\ts a\rfloor$, and $\ve>0$ a fixed constant sufficiently small so
 that \ts $|\mu\cap\nu|<(1-\ve^2)a^2$. For example, $\ve=0.2$ works for~$n$
 large enough.  Since
 $$\ell(\la) \. > \. n-b^2 \. \ge \. (1-\ve^2)\ts a^2 \. > \. |\mu \cap \nu|\ts,
 $$
by Dvir's theorem, we have \ts $g(\la,\mu,\nu)=0$.  The bound on dimension
of $f^\la$ follows by the hook-length formula~\eqref{eq:HLF} and direct calculation.
\end{proof}

\bigskip

\section{Littlewood--Richardson coefficients of $S_n$} \label{s:LR}

\subsection{General inequalities} \label{ss:LR-gen}
The \ts \emph{Littlewood--Richardson
coefficients} \ts $c^\la_{\mu,\nu}$ \ts satisfy:
\begin{equation}\label{eq:LR-groups-def}
\sum_{\la \vdash n} \.  c^\la_{\mu,\ts \nu} \. f^\la \. = \. \binom{n}{k}\ts f^\mu \ts f^\nu\quad \ \ \text{and}  \ \
\quad \sum_{\mu \vdash k, \ts \nu \vdash n-k} \.  c^\la_{\mu,\ts \nu} \. f^\mu \ts f^\nu\. = \.f^\la\ts.
\end{equation}

\begin{lemma}  \label{l:LR-squared-groups}
For every \ts $0\le k \le n$, we have:
$$
\sum_{\la \vdash n} \. \sum_{\mu \vdash k, \ts \nu \vdash n-k} \. \bigl(c^\la_{\mu,\ts \nu}\bigr)^2 \, \ge \,\binom{n}{k} \ts.
$$
\end{lemma}

\begin{proof}  To simplify the notation, let $G=S_n$ and $H=S_k\times S_{n-k}$.  Denote by $\pi^{\mu\nu} = \chi^\mu\otimes\chi\nu$,
and let \ts $\xi^\la = \chi^\la|_H$ \ts be
the restriction of the character $\chi^\la$ to~$H$.
We have:
$$c^\la_{\mu,\ts \nu} \, = \, \sum_{\al \in \Conj(H)}\. z_\al^{-1} \ts \xi^\la(\al) \ts \pi^{\mu\nu}(\al)\ts,
$$
Then:
$$
\aligned
\sum_{\la \vdash n} \. \sum_{\mu \vdash k, \ts \nu \vdash n-k} \. \bigl(c^\la_{\mu,\ts \nu}\bigr)^2  \, & = \,
\sum_{\al,\gamma \in \Conj(H)} \. z_\al^{-1}\ts z_\gamma^{-1} \. \sum_{\la \vdash n} \. \sum_{\mu \vdash k, \ts \nu \vdash n-k} \.  \xi^\la(\al)\ts \xi^\la(\gamma) \.  \pi^{\mu\nu}(\al)\ts \pi^{\mu\nu}(\gamma)\\
& = \, \sum_{\al \in \Conj(H)} z_\al(H)^{-2} \ts \bigl(z_\al(H)\cdot z_\al(G)\bigr) \,  = \, \sum_{\al \in \Conj(H)} \. \frac{z_\al(G)}{z_{\al}(H)}\.,
\endaligned
$$
where \ts $z_\al(H)=|C_H(x)|$ denotes the size of the centralizer of $x \in \al$
within~$H$, and $z_\al(G)=|C_G(x)|$ is the size of the centralizer within~$G$.
Since the \ts RHS~$\ge z_1(G)/z_1(H) = \binom{n}{k}$, we obtain the inequality.
\end{proof}

\begin{lemma}  \label{l:LR-squared-groups-2}
For every \ts $0\le k \le n$, we have:
$$\sum_{\la \vdash n} \.  \bigl(c^\la_{\mu,\ts \nu}\bigr)^2 \, \le \,  \binom{n}{k} \quad \ \text{and} \ \quad
\sum_{\mu \vdash k} \. \sum_{\nu \vdash n-k} \. \bigl(c^\la_{\mu,\ts \nu}\bigr)^2 \, \le \,  \binom{n}{k}\ts.$$
\end{lemma}

\begin{proof}  We have:
$$\sum_{\la \vdash n} \.  \bigl(c^\la_{\mu,\ts \nu}\bigr)^2 \, \le \, \sum_{\la \vdash n} \.
c^\la_{\mu,\ts \nu} \. \frac{f^\la}{f^\mu \ts f^\nu} \, = \,  \frac{1}{f^\mu \ts f^\nu} \ts \cdot \ts f^\mu \ts f^\nu\ts \binom{n}{k}
\, = \,  \binom{n}{k}\ts,
$$
$$\sum_{\mu \vdash k, \ts \nu \vdash n-k} \.  \bigl(c^\la_{\mu,\ts \nu}\bigr)^2 \, \le \, \sum_{\mu \vdash k, \ts \nu \vdash n-k} \.
c^\la_{\mu,\ts \nu} \. \frac{f^\mu \ts f^\nu\ts \binom{n}{k}}{f^\la} \, = \,  \frac{1}{f^\la} \ts \cdot \ts f^\la\ts \binom{n}{k}
\, = \,  \binom{n}{k}\ts,
$$
where were repeatedly use both equations in~\eqref{eq:LR-groups-def}.
\end{proof}

\begin{cor}  \label{c:LR-squared-groups}
For every \ts $0\le k \le n$, we have:
$$
\binom{n}{k} \, \le \,
\sum_{\la \vdash n} \. \sum_{\mu \vdash k} \. \sum_{\nu \vdash n-k} \. \bigl(c^\la_{\mu,\ts \nu}\bigr)^2
\, \le \, \binom{n}{k} \. p(n)\ts.
$$
\end{cor}

Note that \ts $p(k)\ts p(n-k) \ge p(n)$ \ts (cf.~\cite{DP}),
which is why we used here only the first inequality from
Lemma~\ref{l:LR-squared-groups-2}.

\medskip

\subsection{Largest Littlewood--Richardson coefficients}\label{ss:LR-max}
Here we prove theorems~\ref{t:LR-intro-asy} and~\ref{t:LR-intro1}

\begin{prop} \label{p:LR-groups-bounds}
We have:
$$
\frac{1}{\sqrt{p(k) \. p(n-k) \. p(n)}} \. \ts \binom{n}{k}^{1/2} \, \le \,
\rLR(n,k) \, \le \, \binom{n}{k}^{1/2} \ts.
$$
\end{prop}

\begin{proof}
The lower bound follows immediately from Lemma~\ref{l:LR-squared-groups},
while the upper bound follows from Lemma~\ref{l:LR-squared-groups-2}.
\end{proof}

Theorem~\ref{t:LR-intro-asy} follows immediately from here:

\begin{cor}\label{c:LR-asy}
We have:
$$
\binom{n}{k}^{1/2} \ts e^{-d \ts \sqrt{n} \ts (1+o(1))} \. \le  \. \rC(n,k) \. \le  \. \binom{n}{k}^{1/2} \ts,
$$
where \ts $d = \pi (1+\sqrt{2})/\sqrt{6}\approx 3.0963$
\end{cor}

\begin{proof}  Recall that \ts $\log p(n) \sim 2\co\sqrt{n}$, where \ts
$\co=\pi/\sqrt{6}$ defined in~\eqref{eq:VK-const}. By
log-concavity of the partition function~\cite{DP}, we also have:
$$
\log \ts \bigl[p(k) \. p(n-k)\bigr] \, \le \, 2\ts \log p(n/2) \, \sim  \, 2 \. \sqrt{2} \. \co \ts \sqrt{n}\ts.
$$
Now Proposition~\ref{p:LR-groups-bounds} implies the result. \end{proof}

\smallskip

\begin{thm} \label{t:LR-groups-exist1}
Let \ts $0\le k \le n$  and \ts $\la\vdash n$.  Suppose \ts
$f^\mu\ts \ge \ts \sqrt{k!}/a$ \ts and \ts $f^\nu\ts \ge \ts \sqrt{(n-k)!}/a$,
for some \ts $a\ge 1$.  Then there exists
\ts $\la \vdash n$, such that:
$$
f^\la \, \ge \, \frac{\sqrt{n!}}{a^2 \ts p(n)}\., \quad \text{and} \quad
c^\la_{\mu,\ts \nu} \, \ge \, \frac{\sqrt{\binom{n}{k}}}{a^2 \ts p(n)}\,.
$$
\end{thm}

\begin{proof}
Let $\la$ be the index of the largest term in the RHS of
$$
\frac{n!}{a^2 \ts \sqrt{k!\. (n-k)!}} \, \le \,
f^\mu \ts f^\nu \ts \binom{n}{k} \, = \,
\sum_{\la\vdash k} \. c^\la_{\mu,\ts \nu}\. f^\la\ts.
$$
On the one hand, by the upper bound in Proposition~\ref{p:LR-groups-bounds} we have:
$$
f^\la\, \ge \, \frac{f^\mu \ts f^\nu \ts \binom{n}{k}}{p(n)\cdot \rLR(n,k)}  \, \ge \,
\frac{n!}{a^2 \ts \sqrt{k!\. (n-k)!} \. \cdot \.  p(n) \ts \sqrt{\binom{n}{k}}} \, = \,
\frac{\sqrt{n!}}{a^2 \ts p(n)}\..
$$
Similarly,
$$
c^\la_{\mu,\ts \nu} \, \ge \,
\frac{f^\mu \ts f^\nu \ts \binom{n}{k}}{p(n)\cdot \rD(n)}  \, \ge \,
\frac{n!}{a^2 \ts \sqrt{k!\. (n-k)!} \. \cdot \.  p(n) \ts \sqrt{n!}} \, = \,
\frac{\sqrt{\binom{n}{k}}}{a^2 \ts p(n)}\.,
$$
as desired.
\end{proof}

\smallskip

\begin{thm} \label{t:LR-groups-exist}
Let \ts $0\le k \le n$  and \ts $\la\vdash n$.  Suppose \ts
$f^\la\ts \ge \ts \sqrt{n!}/a$, for some \ts $a\ge 1$.  Then there exists
\ts $\mu \vdash k$ and $\nu\vdash (n-k)$, such that:
$$
f^\mu \, \ge \, \frac{\sqrt{k!}}{a \. p(k)\ts p(n-k)}\., \quad
f^\nu\, \ge \, \frac{\sqrt{(n-k)!}}{a \. p(k)\ts p(n-k)}
\quad \text{and} \quad
c^\la_{\mu,\ts \nu} \, \ge \, \frac{\sqrt{\binom{n}{k}}}{a \. p(k)\ts p(n-k)}\,.
$$
\end{thm}

\begin{proof}
Let $(\mu,\nu)$ be the index of the largest term in the RHS of
$$
\sqrt{n!}/a \, \le \, f^\la \, = \, \sum_{\mu\vdash k} \.\sum_{\nu\vdash n-k} \. c^\la_{\mu,\ts \nu}\. f^\mu \ts f^\nu\ts.
$$
On the one hand, by the upper bound in Proposition~\ref{p:LR-groups-bounds} we have:
$$
f^\mu \ts f^\nu\, \ge \, \frac{f^\la}{p(k)\ts p(n-k)\cdot \rLR(n,k)}  \, \ge \,
\frac{\sqrt{n!}/a}{p(k)\ts p(n-k)\cdot \sqrt{\binom{n}{k}}} \, = \,
\frac{\sqrt{k!\ts (n-k)!}}{a \. p(k)\ts p(n-k)}\..
$$
Dividing both sides by $f^\mu$ and $f^\nu$ gives the first two bounds.
Similarly,
$$
c^\la_{\mu,\ts \nu} \, \ge \, \frac{f^\la}{p(k)\ts p(n-k)\cdot f^\mu \ts f^\nu}
\, \ge \,
\frac{\sqrt{n!}/a}{p(k)\ts p(n-k)\cdot \sqrt{k!\ts(n-k)!}}
\, = \,
\frac{\sqrt{\binom{n}{k}}}{a \. p(k)\ts p(n-k)}\,,
$$
as desired.
\end{proof}

\smallskip

Theorems~\ref{t:LR-groups-exist1} and~\ref{t:LR-groups-exist}
immediately imply parts~$\WI$ \ts and~$\WII$ \ts of Theorem~\ref{t:LR-intro1}.
They are patterned after Theorem~\ref{t:Kron-main-groups}. Proof of part~$\WIII$
uses a combination of different tools and is postponed until~$\S$\ref{ss:skew-third-case}.

\medskip

\subsection{Stanley's formula}
Here we use Lemma~\ref{l:LR-squared-groups}
to obtain the LR--analogue of~\eqref{eq:Kron-squares} due to Stanley,
as his proof is unpublished.  First, we need the following recent
identity of Harris and Willenbring:

\begin{lemma}[\cite{HW}] \label{t:HW}
We have:
\begin{equation}
\label{eq:LR_squares}
\sum_{\la,\ts\mu,\ts\nu \ts \vdash n} \, \bigr(c^\la_{\mu,\nu}\bigr)^2 \. q^{|\mu|} \. t^{|\nu|}
\,= \, \prod_{i=1}^\infty \frac{1}{1-q^i -t^i}\..
\end{equation}
\end{lemma}

\begin{proof} Our proof is built on the proof of Lemma~\ref{l:LR-squared-groups}.
Recall that
$$
z_\al \. := \, \prod_i \. i^{m_i(\al)}\. m_i(\al)!
$$
for all $\al \vdash n$, where $m_i(\al)$ denote the number of parts of size $i$ in~$\al$.
We then have:
$$
\aligned
& \sum_{\la\vdash n, \ts \mu\vdash k, \ts \nu\vdash n-k} (c^\la_{\mu,\nu})^2 \, = \,
\sum_{\al \vdash k, \be \vdash n-k} \. \frac{z_{\al\cup\be}}{z_\al \cdot z_\be} \, = \,
\sum_{\al \vdash k, \ts \be \vdash n-k} \, \prod_{i\ge 1} \. \frac{i^{m_i(\al\cup\be)}\ts m_i(\al\cup\be)!}{i^{m_i(\al)}\ts m_i(\al)! \cdot i^{m_i(\be)}\ts m_i(\be)!}
\\
& \hskip1.0cm = \, \sum_{\al \vdash k, \ts \be \vdash n-k} \, \prod_{i\ge 1} \. \binom{ m_i(\al)+m_i(\be)}{m_i(\al)} \, = \, [q^k t^{n-k}] \prod_{i\ge 1} \frac{1}{1-q^i - t^i}\..
\endaligned
$$
Note that the last line follows from \ts $m_i(\al\cup\be) = m_i(\al)+m_i(\be)$.
\end{proof}

\begin{cor}[{Stanley~\cite[Exc.~$7.79$]{EC2-supp}}]\label{c:LR-stanley}
$$
\sum_{\la,\ts\mu,\ts\nu \ts \vdash n} \, \bigr(c^\la_{\mu,\nu}\bigr)^2 \, \sim  \, K\ts 2^{n} \quad {as} \ \ n\to \infty,
$$
where
$$
K \, = \, \prod_{i=1}^\infty \. \left(1-\frac{1}{2^i}\right)^{-1} \approx \. 3.4627466195
$$
\end{cor}

Note that this asymptotic bound implies part~$(\ast\ast)$ of
Stanley's theorem~\ref{t:stanley-asy}.  On the other hand,
neither Theorem~\ref{t:LR-groups-exist} implies the corollary,
nor vice versa.

\begin{proof}[Proof of Corollary~\ref{c:LR-stanley}]
Let $p_2(n)$ is the number of \emph{bicolored partitions} of~$n$ \ts defined as follows
$($see \ts \cite[\href{https://oeis.org/A070933}{A070933}]{OEIS}$):$
$$\sum_{n=0}^\infty \. p_2(n) \ts t^n  \, = \, \prod_{i=1}^\infty \frac{1}{1-2t^i}\,,
$$
and recall that \ts $p_2(n) \sim K \ts 2^n$.
Taking $q=t$ in Lemma~\ref{t:HW}, we obtain the result.
\end{proof}

\medskip

\subsection{Largest refined LR--coefficients}
By analogy with the Kronecker coefficients, define

\begin{equation}\label{eq:def-LR-ref}
\rC(\la) \, := \, \max_{0 \le k \le n} \. \max_{\mu \vdash k}\. \max_{\nu \vdash n-k}
 \. c^\la_{\mu,\nu}\..
\end{equation}

\begin{thm} \label{t:LR-refined-upper}
Fix $\la\vdash n$ and let $\ell\ge \ell(\la)$.  Then:
\begin{equation}\label{eq:thm-LR-ref}
\rC(\la)^2  \, \le \, \prod_{(i,j)\in \la} \. \frac{2\ts \ell +j-i}{\ell +j-i}\..
\end{equation}
\end{thm}

Letting $\ell\to \infty$ gives Stanley's upper bound $\rC(\la)\le 2^{n/2}$,
but gives a much better bound for smaller~$\ell$ (see below).

\begin{proof}  We have:
$$\aligned
s_{\lambda}(1^{2 \ell}) \, & = \, \sum_{\mu, \nu} \. c^{\lambda}_{\mu, \nu} \. s_{\mu}(1^{\ell}) \. s_{\nu}(1^{\ell})
\, = \, \sum_{\rho, \mu, \nu} \. c^{\lambda}_{\mu, \nu} \. c^{\rho}_{\mu, \nu} \. s_{\rho}(1^{\ell}) \\
& \ge \, s_{\lambda}(1^{\ell}) \. \sum_{\mu, \nu} \. (c^{\lambda}_{\mu, \nu})^2 \, \ge \,  s_{\lambda}(1^{\ell}) \.  \rC(\la)^2\ts.
\endaligned
$$
By the hook-content formula~\eqref{eq:HCF}, we conclude:
$$
\rC(\la)^2 \, \le \, \frac{s_\la\bigl(1^{2\ell}\bigr)}{s_\la\bigl(1^{\ell}\bigr)} \, = \.
\left[\prod_{(i,j)\in \la} \. \frac{2\ell +j-i}{h_{ij}}\right] \cdot \left[\prod_{(i,j)\in \la} \. \frac{\ell +j-i}{h_{ij}}\right]^{-1}
\. = \, \prod_{(i,j)\in \la} \. \frac{2\ts \ell +j-i}{\ell +j-i}\.,
$$
as desired. \end{proof}

\begin{cor} \label{c:LR-ref-rows-upper}
For all $\la\vdash n$ and $\ell=\ell(\la)$, $m = \lambda_1$ we have:
$$
\rC(\la) \, \le \,  (m+\ell)^{\ell^2/2}\ts. 
$$
\end{cor}

\begin{proof}  For a fixed $\ell=\ell(\la)$ and $n=|\la|$, the product
in~\eqref{eq:thm-LR-ref} maximizes when \ts $m=\la_1 = \lceil n/\ell\rceil$.
Indeed,  moving the squares below and to the left decreases the contents $(j-i)$,
and thus increases the ratios \ts $(2\ell +j-i)/(\ell +j-i)$ \ts in the product.
Thus, we have:
$$
\rC(\la)^2  \, \le \, \prod_{i=1}^\ell \.\prod_{j=1}^m \. \frac{2\ts \ell +j-i}{\ell +j-i}
\, \le \, \prod_{i=1}^\ell \.\prod_{j=m-\ell+1}^m \bigl(2\ts \ell +j-i\bigr)
\, \le \, (m+\ell)^{\ell^2}\ts. 
$$
This implies the result.
\end{proof}

\begin{rem}{\rm
It follows from the proof that the RHS of~\eqref{eq:thm-LR-ref} maximizes
at a rectangle~$\la=m^\ell$ (cf.~\cite{BG}). Note, however, that $c^\la_{\mu,\nu}\le 1$
in this case, see e.g.~\cite[$\S$4.3]{PP-unim}.
}\end{rem}

\medskip

\subsection{Largest LR--coefficient with few rows} Define

$$
\rC_\ell(n) \, := \, \max_{\la\vdash n\ts, \, \ell(\la)=\ell}
\, \max_{\mu \vdash k}\, \max_{\nu \vdash n-k}
 \, c^\la_{\mu,\nu}\..
$$


\begin{thm} \label{t:LR-refined-row}
For all \ts $\ell, n \ge 1$, we have:
$$
n^{\frac12 \ts \ell^2 \. - \. a\ell} \, \. e^{-b\ts \ell^2\log \ell} \, \le \,
\rC_{\ell}(n) \, \le \, (n+1)^{\frac12 \ts \ell^2},
$$
for some universal constants \ts $a, b>0$.
\end{thm}


\begin{cor} Let \ts $\{\ell_n\}$ \ts be an integer sequence which satisfies
\ts $\ell_n = O(\sqrt{n}/\log n)$ \ts and \ts $\ell_n=\om(1)$.  Then we have:
$$
\log \rC_{\ell_n}(n) \, \sim \, \frac12 \.(\ell_n)^2 \. \log n  
\ts.
$$
\end{cor}

\smallskip

\begin{proof}[Proof of Theorem~\ref{t:LR-refined-row}]
The upper bound follows from Corollary~\ref{c:LR-ref-rows-upper} since $m + \ell \le n+1$.
The following argument is more direct and gives a slightly
better bound; we include it for completeness.

Recall the combinatorial interpretation for \ts $c^\la_{\mu,\nu}$ \ts
in terms of the number of \emph{Knutson--Tao hives} \cite{KT1}
(see also~\cite{PV} and \emph{BZ-patterns} in~\cite[$\S$7.A1.3]{EC2}).
These are integer triangular tables of size $(\ell+1)$,
with entries $\le n$ and fixed boundary given by $\la,\mu$ and~$\nu$.
This leaves only $\binom{\ell-1}{2}$  entries, which immediately
implies the upper bound
$$
\rC_\ell(n) \, \le \, (n+1)^{(\ell-1)(\ell-2)/2}\ts.
$$

For the lower bound, recall the combinatorial interpretations for \ts
$c^\la_{\mu,\nu}$ \ts in terms of \emph{honeycombs}~\cite{KT1}
and in terms of the \emph{Knutson--Tao puzzles} \cite{KT}.
The latter are triangles of size $(m+\ell)$, where $m=\la_1$, which are tiled
with pieces of 3 different types (up to rotation and parallel translation).
Note that these puzzles are uniquely determined by the positions of the
$(111)$--triangles, since the lozenges extend the rows of $1$-s between
them and $(000)$--triangles fill the remaining space.

Consider the \emph{honeycomb graph} of $1$-edges in the graph $\Ga$ dual
to the puzzle graph.  Fix $\Ga$ with $r=\binom{\ell}{2}$ regular hexagons
of size $m/(3\ell)$.  Let the faces of $\Ga$ ``breathe'', i.e.\
contract by at most $m/(18\ell)$. The changes to honeycombs are small enough,
so that none of them degenerate (since none of the edges degenerate).
Pin one of the central vertices of the honecomb at a fixed position
in the center of the triangle.  All honeycombs will fit the $(m+\ell)$
triangle, and the resulting puzzles are distinct then.
Their number is at least
$$
\left(\frac{m}{18\ts \ell}\right)^{r} \, \ge \,
\left(\frac{n}{18\ts \ell^2}\right)^{r}\ts.
$$
Since the total number of triples of partitions $(\la,\mu,\nu)$ is at most
$$\binom{n}{\ell}^3 \. \le \. n^{3\ell},
$$
at least one triple has \ts $\ge n^{r-3\ell}\ts (18\ts \ell^2)^{-r}$ \ts puzzles, as desired.
\end{proof}

\medskip

\subsection{Vanishing LR--coefficients}
We believe the natural analogue of Conjecture~\ref{conj:kron-sharp} fails for
the LR--coefficients.

\begin{conj}\label{conj:LR-sharp}
There exists three Plancherel partitions sequences
$\{\la^{(n)}\vdash n\}$, $\{\mu^{(n)}\vdash n/2\}$ and $\{\nu^{(n)}\vdash n/2\}$, s.t.
$$
\frac{1}{\sqrt{n}} \left(\frac{n}{2} \. - \. \log_2 c^{\la^{(n)}}_{\mu^{(n)},\. \nu^{(n)}}\right) \. \to \. \infty \ts.
$$
\end{conj}

The following result gives a strong evidence in favor of the conjecture.

\begin{prop}\label{p:non-robust}
Let $\{\mu^{(n)}\vdash n/2\}$, $\{\nu^{(n)}\vdash n/2\}$ be
two Plancherel partitions sequences.  Then there exist a partition sequence
$\{\la^{(n)}\vdash n\}$ with
\begin{equation}\label{eq:robust}
f^{\la^{(n)}} \, = \, \sqrt{n!} \, e^{O(\sqrt{n}\ts \log n)}
\end{equation}
such that:
$$
c^{\la^{(n)}}_{\mu^{(n)},\. \nu^{(n)}} \. = \. 0\ts.
$$
\end{prop}

The proposition should be compared with part~$\WI$ of
Theorem~\ref{t:LR-intro1}.  The notable difference in
the statement is the $(\log n)$ factor in the error term.

\begin{proof}
For a partition $\la = (\la_1,\la_2,\ldots) \vdash n$, let
$$\wh\la \. := \. (\la_3,\la_4,\ldots) \cup 1^{\la_1+\la_2}
$$
Fix a Plancherel partition sequence $\{\la^{(n)}\}$.   Then $\bigl\{\wh\la^{\ts (n)}\!\bigr\}$
satisfies~\eqref{eq:robust}, and has \ts $\ell\bigl(\la^{(n)}\bigr)\sim 6\sqrt{n}$ \ts parts.
Since $\{\mu^{(n)}\}$, $\{\nu^{(n)}\}$ are Plancherel, they have
\ts $2\sqrt{n} \bigl(1+o(1)\bigr)$ \ts parts.  Since $\{\mu^{(n)}\circ\nu^{(n)}\}$
have at most \ts $4\sqrt{n} \bigl(1+o(1)\bigr)$ \ts parts, this implies the result.
\end{proof}

\medskip

\subsection{Containment of the largest LR--coefficients}\label{ss:LR-cont}
We start with the following beautiful result by Lam, Postnikov and Pylyavskyy~\cite{LPP}
which we state in the following equivalent form.

\begin{thm}[\cite{LPP}] \label{t:LR-modularity}
Let $\lambda \vdash n$, $\mu \vdash k$, $\nu\vdash n-k$.  Denote $\al=\mu\cap \nu$, $\be = \mu\cup\nu$.
Then
$$c^\la_{\mu,\nu} \. \le \. c^\la_{\al,\be}\..
$$
\end{thm}

This implies that the largest LR--coefficient is attained at a flag of partitions.

\begin{cor}\label{c:LR-modularity}
Let $\lambda \vdash n$. There exist \ts $\al \subseteq \be\ssu \la$, \ts $|\al|+|\be|=n$, s.t.
$$\rC(\la) \. = \. c^{\lambda}_{\al, \be}\..
$$
In particular, for every $n$, there exist \ts $\al \subseteq \be \ssu \la$, \ts $|\al|+|\be|=|\la|=n$,
s.t.
$$\rC(n) \. = \. c^{\lambda}_{\al, \be}\..
$$
\end{cor}

\begin{rem}{\rm
Note that for some $n\ge 10$, the sequence \ts $\{\rC(n,k), 0\le k \le n\}$ \ts
is not unimodal (see~\ref{a:LR}).  Corollary~\ref{c:LR-modularity} suggests an
explanation why not: the unimodality would
imply that for $k=n/2$ we have $\al=\be$.  Denote by $\zeta(n)$ the smallest $k$
s.t. $\rC(n,k)=\rC(n)$.  Let $\rho(n):= n/2-\zeta(n)$.
For example, $\zeta(18) = 7$ since $\rC(18,7) = \rC(18) =11$, so $\rho(18)=2$.

What can be said about the asymptotics of $\rho(n)$?
Theorem~\ref{t:LR-intro-asy} implies that \ts $\rho(n) = o(n)$.
It would be interesting to find a nontrivial lower bound on $\rho(n)$.
Finally, we make the following conjecture based on our computer experiments.

\begin{conj}\label{conj:LR-sub}
Let $c^\la_{\mu,\nu} = \rC(n)$, where $\la\vdash n$.  Then $\mu \subseteq \nu$
or $\nu \subseteq \mu$.
\end{conj}

In other words, we are claiming that ``there exist'' clause in the
corollary can be replaced with ``for all''. For example, for $n=18$ the
maximum $\rC(18)=11$ is attained at  $\la=75321$, $\mu = 5321$
and $\nu = 421$ and its conjugates and transpositions:
$$
c^\la_{\mu,\nu} \. = \. c^\la_{\nu,\mu} \. = \. c^{\la'}_{\mu',\ts\nu'} \. = \. c^{\la'}_{\nu',\ts\mu'}  \. = \. 11\ts.
$$
Here $\la \sss \mu \sss \nu$.
}\end{rem}

\medskip

\subsection{Monotonicity and stability of the largest LR--coefficients}\label{ss:LR-mono}
We conclude this section with the following results suggested by
numerical values for $\rC(n,k)$ in the appendix~$\S$\ref{a:LR}.

\begin{prop}  \label{p:LR-mono}
The sequence \ts $\bigl\{\rC(n)\bigr\}$ \ts is nondecreasing.
Similarly, for every fixed $k\ge 1$, the sequence
\ts $\bigl\{\rC(n,k)\bigr\}$ \ts is bounded and nondecreasing.
In particular, for every fixed $k$, the sequences
$\bigl\{\rC(n,k)\bigr\}$ \ts stabilizes when $n$ is large enough.
\end{prop}

\begin{proof}
Recall the usual Young tableaux combinatorial interpretation
of the LR--coefficients, see e.g.~\cite{EC2}.
It is well known and easy to see that
$$
c^\la_{\mu,\nu} \. \le \. c^{\la+1}_{\mu,\nu+1}\ts,
$$
where \ts $(\la_1,\la_2,\la_3,\ldots) \ts +\ts 1 \ts :=\ts
(\la_1+1,\la_2,\la_3,\ldots)$.
By fixing $\mu$, we conclude that $\rC(n)\le \rC(n+1)$ and
$\rC(n,k)\le \rC(n,k+1)$, for all $n\ge 1$, $k\ge 0$.

To prove that $\rC(k,k), \rC(k+1,k), \rC(k+2,k), \ldots$ is bounded,
observe that
$$
c^{\la}_{\mu, \nu} \, \le \, \frac{f^{\la/\nu}}{f^{\mu}} \, \le \, k!\.,
$$
for all $\mu\vdash k$. This implies the result.
\end{proof}

\smallskip

\begin{thm}\label{t:LR-stab}
For every \ts $k\ge 0$ \ts and \ts $n \ge \binom{k+1}{2}$, we have
\. $\rC(n,k)  \ts = \ts \rD(k)$.
\end{thm}

\begin{proof}
By Proposition~\ref{p:LR-mono}, we have \ts
$ \rC(k,k) \le \rC(k+1,k) \le \rC(k+2,k) \le \ldots$ \ts
Recall that the $S_k$-module \ts $\bss^{\la/\nu}$ corresponding to
skew shape $\la/\nu$ is a submodule of the regular representation
(see e.g.~\cite{Sag}). Equivalently,
$$
(e_1)^n \. - \. s_{\la/\mu} \. \ge 0
$$
is \emph{Schur positive} (cf.~\cite{BBR,LPP}).  This implies
that \ts $c^\la_{\mu,\nu} \le f^\nu$. We conclude that
\ts $\rC(n,k) \le \rD(k)$ for all $n\ge k$.

Denote
$\de_k=(k-1,k-2,\ldots,1)$ and let $r=n-\binom{k+1}{2}$.
Define $\la=\de_{k+1}+(r)$, $\nu=\de_{k}+(r)$, and observe that
$\la/\nu$ is a disjoint union of $k$ squares.  Then \ts $\bss^{\la/\nu}$
is regular and \ts $\rC(n,k) = \rD(k)$ \ts
for all \ts $n\ge \binom{k+1}{2}$.
\end{proof}

\smallskip

\begin{rem}{\rm
These results on stability of the largest LR--coefficients are
modeled after similar results on stability of certain Kronecker
coefficients (see e.g.~\cite{BOR,Man,Sam}).  Of course, the latter
results are more technical since there is no combinatorial
interpretation for Kronecker coefficients.

By the theorem, \ts $\bigl\{\rC(n,k)\bigr\}$ \ts stabilizes at or
before $\binom{k+1}{2}$.  The data in the appendix for \ts $0\le k\le 6$,
suggests that it stabilizes precisely at \ts $\binom{k+1}{2}$.
We conjecture that this is true for all \ts $k\ge 0$.
}\end{rem}

\bigskip

\section{Large LR--coefficients imply large dimensions} \label{s:LR-imply-dim}

The goal of this section is obtain  Theorem~\ref{t:LR-upper-dim}, which is
the LR--analogue of the bound \ts $f^\la \ge g(\la,\mu,\nu)$ \ts
in~\eqref{eq:Kron-upper-groups}.

\subsection{Upper bounds}  To simplify presentation, here and
everywhere in this section we use (finite) sums over partitions
for which the corresponding terms are well-defined and nonzero.

\begin{lemma} Let $\mu\vdash k$ and $\nu\vdash (n-k)$. Then
$$
\sum_{\lambda\vdash n} \. \left( c^{\lambda}_{\mu, \nu}\right)^2 \, = \, \sum_{\alpha, \beta,
\gamma, \delta} \. c^{\mu}_{\alpha, \gamma} \. c^{\mu}_{\alpha, \delta} \.
c^{\nu}_{\beta, \gamma} \. c^{\nu}_{\beta, \delta}\ts.
$$
\end{lemma}

\begin{proof}   Let $\mu, \nu, \tau, \ka$ be fixed partitions.
Start with the \emph{skew Cauchy identity} for Schur functions:
$$
\sum_{\lambda} \. s_{\lambda/\mu}(\mathbf{x}) \. s_{\lambda/\tau}(\mathbf{y})
\, = \, \prod_{i=1}^\infty \. \prod_{j=1}^\infty \, \frac{1}{1 - x_i y_j} \,
\sum_{\alpha} \. s_{\tau/\alpha}(\mathbf{x}) \. s_{\mu/\alpha}(\mathbf{y})\ts.
$$
(see e.g.~\cite{SaS}).  We expand it as follows:
$$
\aligned
\sum_{\lambda, \nu, \tau} \, c^{\lambda}_{\mu, \nu} \. s_{\nu}(\mathbf{x})
\. c^{\lambda}_{\tau, \ka} \. s_{\ka}(\mathbf{y}) \, &= \, \sum_{\beta} \.
s_{\beta}(\mathbf{x}) \. s_{\beta}(\mathbf{y}) \,
\sum_{\alpha, \gamma, \delta} \, c^{\tau}_{\alpha, \gamma} \.
s_{\gamma}(\mathbf{x}) \. c^{\mu}_{\alpha, \delta}
s_{\delta}(\mathbf{y}) \\
&= \, \sum_{\alpha, \beta, \gamma, \delta, \nu, \tau} \, c^{\tau}_{\alpha, \gamma}
\. c^{\mu}_{\alpha, \delta} \. c^{\nu}_{\beta, \gamma} \. c^{\ka}_{\beta, \delta}
\. s_{\nu}(\mathbf{x}) \. s_{\ka}(\mathbf{y})\ts.
\endaligned
$$
The needed identity follows then by taking the coefficients at \ts
$s_{\nu}(\mathbf{x}) \. s_{\ka}(\mathbf{y})$ \ts from both sides:
$$
\sum_{\lambda} \. c^{\lambda}_{\mu, \nu} \. c^{\lambda}_{\tau, \ka} \, = \,
\sum_{\alpha, \beta, \gamma, \delta} \. c^{\tau}_{\alpha, \gamma} \.
c^{\mu}_{\alpha, \delta} \. c^{\nu}_{\beta, \gamma} \. c^{\ka}_{\beta, \delta}\ts.
$$
Taking \ts $\tau=\mu$ \ts and \ts $\ka = \nu$ \ts implies the result.
\end{proof}

\begin{lemma} \label{l:LR-ind-step}
Let $\la\vdash n$, $\mu\vdash k$ and $\nu\vdash (n-k)$.  Then:
$$
c^{\lambda}_{\mu, \nu} \,  \le \, \sqrt{n} \. p(k) \ts p(n-k) \. \rC(\mu) \ts \rC(\nu)\ts.
$$
\end{lemma}

\begin{proof} We have:
$$\aligned
\bigl(c^{\lambda}_{\mu, \nu}\bigr)^2 \, & \le \,
\sum_{\la\vdash n} \. \left( c^{\lambda}_{\mu, \nu}\right)^2 \, = \, \sum_{\alpha, \beta,
\gamma, \delta} \. c^{\mu}_{\alpha, \gamma} \.  c^{\mu}_{\alpha, \delta} \.
c^{\nu}_{\beta, \gamma} \. c^{\nu}_{\beta, \delta} \, \\
& \le \, \sum_{a=0}^{\min\{k,n-k\}} \.p(k-a) \. p(n-k-a) \. p(a)^2 \,  \rC(\mu)^2 \. \rC(\nu)^2\\
& \le \, n\. p(k)^2 \. p(n-k)^2 \,  \rC(\mu)^2 \. \rC(\nu)^2\ts,
\endaligned
$$
which implies the result.
\end{proof}

\smallskip

\begin{thm}\label{t:LR-upper-dim}
Let $\lambda \vdash n$, $|\mu|+|\nu|=n$ be fixed partitions. Then:
\begin{equation}\label{key}
f^{\lambda} \,\ge \,  e^{- u \ts n} \. \left(c^{\lambda}_{\mu, \nu} \right)^{\log_2 n}\ts,
\end{equation}
where $u>0$ is a universal constant.
\end{thm}

\begin{proof}
The idea is to iteratively apply the inequalities
\begin{equation}\label{eq:iq0}
f^{\lambda} \, \ge \, f^{\mu} f^{\nu} c^{\lambda}_{\mu, \nu}
\end{equation}
and the following corollary of Lemma~\ref{l:LR-ind-step}.
\begin{equation}\label{eq:iq}
c^{\lambda}_{\mu, \nu} \,  \le \, e^{a\ts \sqrt{n}} \. \rC(\mu) \ts \rC(\nu)\ts,
\end{equation}
for some universal constant $a>0$.

Consider a {\it full binary tree} \ts $T = T(\lambda, \mu, \nu)$ of depth $m = \lceil \log_2 n \rceil$
defined as follows.  Let the nodes be labeled by partitions according to the following rule:
\begin{itemize}
\item[(i)] \. $\lambda$ is the root with children $\mu$ and $\nu$;
\item[(ii)] \. every non-leaf node \ts $\rho$ \ts has children \. $\alpha, \beta$ \. so that \.
${c^{\rho}_{\alpha, \beta}} \ts = \ts \max_{\phi, \psi} {c^{\rho}_{\phi, \psi}}$ \\ (these could be empty partitions).
\end{itemize}
For each node $\rho \in T$, denote by $L(\rho)$ and $R(\rho)$ its children in~$T$,
and let $k_{\rho} := |\rho|$ be the size of the partition $\rho$. Iterating the
inequality~\eqref{eq:iq0} along~$T$ starting from the root $\lambda$ we obtain:
\begin{equation}\label{flt}
{f^{\lambda}} \, \ge \, {f^{L(\lambda)}} \ts {f^{R(\lambda)}} \. {c^{\lambda}_{L(\lambda), R(\lambda)}} \, \ge \,
\cdots \, \ge \, \prod_{\rho \in T} \. c^{\rho}_{L(\rho), R(\rho)}\ts.
\end{equation}
From  \eqref{eq:iq} we have the following local inequalities:
$$
c^{\alpha}_{L(\alpha), R(\alpha)} \. c^{\beta}_{L(\beta), R(\beta)} \, \ge \,
e^{-a \sqrt{k_{\rho}}} c^{\rho}_{\alpha, \beta}\,, \ \, \text{ where } \ \.
\rho \in T, \, \rho \vdash k_\rho, \, \alpha = L(\rho), \, \beta = R(\rho)\ts.
$$
Applying these inequalities in~\eqref{flt}  repeatedly, starting from the
bottom $(m-1)$-st level to top (the root at $0$-th level), we obtain:
\begin{equation}\label{f_inequality}
f^{\lambda} \, \ge  \, \prod_{\rho \in T} \. c^{\rho}_{L(\rho), R(\rho)} \, \ge \,
\left[\prod_{i = 0}^{m-1} \. \prod_{\rho \in T_i} e^{-a \sqrt{k_{\rho}} \. (m -1- i)}\right]
\. \left(c^{\lambda}_{\mu, \nu}\right)^m\ts,
\end{equation}
where $T_i$ denotes the $i$-th level of $T$. By construction, we have \ts
$\sum_{\rho \in T_i} k_{\rho} = n$ \ts for every level \ts $0 \le i \le m$,
with \ts $|T_i|= 2^i$. Using the Cauchy--Schwarz inequality we get
$$
\sum_{\rho \in T_i} \. \sqrt{k_{\rho}} \,\. \le \, \sqrt{2^i \. \sum_{\rho \in T_i} k_{\rho}}
\, = \, \sqrt{2^i } \. \sqrt{n}\..
$$
This implies
$$\prod_{\rho \in T_i} \. e^{-a\ts \sqrt{k_{\rho}}\. (m -1- i)} \, \ge \,
e^{-a \ts \sqrt{n} \. \sqrt{2^i}\. (m-1- i)}\..
$$
Note that
$$
\sum_{i = 0}^{m-1} \. \sqrt{2^i} \. (m-1 - i) \, = \, (3 +2 \sqrt{2}) (\sqrt{2^m} - 1)
\. -\. m(1+\sqrt{2}) \, = \, O(\sqrt{2^m}) \, = \, O(\sqrt{n}).
$$
From here and~\eqref{f_inequality}, we conclude that
$$
f^{\lambda} \, \ge \, \exp\left( {-a\ts \sqrt{n} \. \sum_{i = 0}^{m-1} \. \sqrt{2^i}\. (m -1- i)} \right)
\. \left(c^{\lambda}_{\mu, \nu}\right)^m \, \ge \, e^{-u \ts n} \. \bigl(c^{\lambda}_{\mu, \nu} \bigr)^{\log_2 n}\ts,
$$
for some universal constant $u \ge 0$, as desired.
\end{proof}

\medskip

\subsection{Proof of Theorem~\ref{t:LR-intro2}}  By assumption, $k_n=\lfloor n/2\rfloor$.
From the inequality \eqref{key} we conclude that
$$
\log f^{\lambda^{(n)}} \, \ge \,  (\log_2 n) \cdot \log \left[2^{n/2} \. e^{O(n/\log n)}\right] \. - \. O(n)
\, = \,  \frac{1}{2} \. n \log n \. - \. O(n).
$$
Using the inequality
$$
f^{\nu} \, \ge \, \frac{f^{\lambda} \ts c^{\lambda}_{\mu \nu}}{f^{\mu} \ts \binom{n}{k}}
$$
we obtain:
$$\aligned
\log f^{\nu^{(n)}} \, & \ge \, \log f^{\la^{(n)}} \. + \.
\log c^{\la^{(n)}}_{\mu^{(n)},\.\nu^{(n)}} \. - \. \log f^{\mu^{(n)}}
\. - \. \log \binom{n}{k_n}\\
 &\ge \, \frac{1}{2} \. n \log n \. - \. O(n) \. + \. \log \binom{n}{n/2} \. - \. O(n/\log n) \\
& \qquad \ - \. \frac{1}{2} \. (n/2) \ts \log \ts (n/2) \. + \. O(n) \. - \. \log \binom{n}{n/2} \\
 &\ge \, \frac{1}{2} \. (n/2) \log \ts (n/2) \. - \. O(n) \, = \, \log \sqrt{(n/2)!} \. - \. O(n).
\endaligned
$$
The inequality for $f^{\mu^{(n)}}$ follows similarly. \ $\sq$

\medskip

\begin{rem}\label{rmk:ineq}\rm
Stronger versions of the inequality \eqref{key} would give better bounds for
Theorem~\ref{t:LR-intro2}.  Notably, suppose~\eqref{key} holds with $u = 1/2$, then
$$
c^{\la^{(n)}}_{\mu^{(n)}, \nu^{(n)}} \, = \, 2^{n/2} \. e^{o(n/\log n)}
$$
imply
$$
f^{\la^{(n)}}  \ge \sqrt{n!}\, e^{o(n)}  \quad \text{and} \quad
f^{\mu^{(n)}}, \, f^{\nu^{(n)}} \, \ge \, \sqrt{(n/2)!}\, e^{o(n)}\..
$$
When $\{\la^{(n)}\}$, $\{\mu^{(n)}\}$ and $\{\nu^{(n)}\}$ have strongly
stable limit shapes, this would imply that all three partition sequences have
VKLS shape by Theorem~\ref{t:stable}.

\begin{conj}\label{conj:LR-upper-dim}
Let $\lambda \vdash n$ and $\mu,\nu\vdash n/2$ be fixed partitions.
Let
$$
\wt f^{\,\lambda} \. := \. \frac{f^{\lambda}}{\sqrt{n!}} \quad \text{and} \quad
\wt c^{\lambda}_{\mu, \nu} \. := \. \frac{c^{\, \lambda}_{\mu, \nu}}{\sqrt{\binom{n}{n/2}}}\..
$$
Then:
\begin{equation*}
\wt f^{\,\lambda} \,\ge \,  a \left(\wt c^{\,\lambda}_{\mu, \nu}\right)^{u\ts\log n}\ts,
\end{equation*}
where \ts $a,\ts u>0$ \ts are universal constants.
\end{conj}

This is a strong conjecture.  If true, then
$$
c^{\la^{(n)}}_{\mu^{(n)}, \nu^{(n)}} \, = \, 2^{n/2} e^{O(\sqrt{n})}
$$
imply
$$
f^{\la^{(n)}}  \ge \sqrt{n!}\, e^{O(\sqrt{n}\. \log n)}  \quad \text{and} \quad
f^{\mu^{(n)}}, \, f^{\nu^{(n)}} \, \ge \, \sqrt{(n/2)!}\, e^{O(\sqrt{n}\. \log n)}\..
$$

\end{rem}

\bigskip

\section{Standard Young tableaux of skew shape}\label{s:skew}

\subsection{Sum of squares}  \label{ss:skew-sum-squares}
We start our study of \ts $f^{\la/\mu}=\#\SYT(\la/\mu)$ \ts
with the following technical result.

\begin{lemma} \label{l:SYT-bounds}
For every \ts $0\le m \le n$, we have:
$$
\frac{(n-1)!}{(m-1)!} \, \le \, \sum_{\la \vdash n} \, \sum_{\mu \vdash m} \. \bigl(f^{\la/\mu}\bigr)^2
\, \le \, \frac{n!}{m!} \. p(m)\ts.
$$
\end{lemma}

\begin{proof} We start with the following Stanley's identity given in~\cite{Stanley-diff-posets}:
$$\sum_{m=0}^\infty \,\sum_{n=0}^\infty \, \sum_{\mu\vdash m}  \, \sum_{\la\vdash m+n} \, \bigl(f^{\la/\mu}\bigr)^2 \. \frac{q^m \ts t^n}{n!}
\, = \, \frac{1}{1-t/(1-q)} \. \prod_{i=1}^\infty \. \frac{1}{1-q^i}\..
$$
Taking the coefficient \ts $\bigl[q^m\ts t^{n-m}\bigr]$ \ts on both sides gives:
\begin{equation}\label{eq:sum-squares-skew-shape}
\sum_{\la \vdash n} \,\sum_{\mu \vdash m} \. \bigl(f^{\la/\mu}\bigr)^2 \,  = \,  (n-m)! \. \sum_{k=1}^m \. \binom{n-m+k-1}{k-1} \ts p(m-k).
\end{equation}
Bounding \ts $p(m-k)\leq p(m)$ \ts and summing the binomials give the upper bound:
$$
\sum_{\la \vdash n} \, \sum_{\mu \vdash m} \. \bigl(f^{\la/\mu}\bigr)^2  \,  \le \,  (n-m)! \.\ts \binom{n}{m} \. p(m)\ts.
$$
Finally, the term \ts $k=m$ \ts in the RHS of~\eqref{eq:sum-squares-skew-shape} gives the lower bound.
\end{proof}

\medskip

\subsection{Largest $\#\SYT$s of skew shape}\label{ss:skew-largest}
For \ts $0\le m \le n$, define
$$
\rF(m,n) \, := \, \max_{\la \vdash n} \. \max_{\mu \vdash m} \. f^{\la/\mu}\..
$$

\begin{cor}\label{c:SYT-asy}
Let $1 \le m \le n$, and $n\to \infty$.  Then:
$$
\sqrt{\frac{n!}{m!}} \,\. e^{-2\co\sqrt{n}\ts (1+o(1))}\,\le \, \rF(m,n)\,
\le \, \sqrt{\frac{n!}{m!}} \, \. e^{2\co\sqrt{n}\ts (1+o(1))}\ts.
$$
\end{cor}

\begin{proof} Let \ts $1\le m \le n$.  From Lemma~\ref{l:SYT-bounds}, we immediately have:
\begin{equation}\label{eq:SYT-max}
\sqrt{\frac{(n-1)!}{(m-1)!}} \,\frac{1}{\sqrt{p(m) \. p(n)}} \. \, \le \, \rF(m,n)\, \le \, \sqrt{\frac{n!}{m!}} \, \sqrt{p(m)}\..
\end{equation}
These bounds imply the result.
\end{proof}


\medskip

\subsection{Skew shapes with the largest $\#\SYT$s}\label{ss:skew-shapes-largest}

\begin{thm}\label{t:skew-Plancherel}
Let \ts $m_n := \lfloor\theta n\rfloor$ for some $0<\theta<1$, and
let $\{\la^{(n)}\vdash n\}$, $\{\mu^{(n)}\vdash m_n\}$ be two Plancherel
partitions sequences.  Then:
\begin{equation} \label{eq:SYT-thm}
f^{\la^{(n)}/\mu^{(n)}} \, = \, \sqrt{\frac{n!}{m_n!}} \,  e^{-O\bigl(n^{2/3}\ts\log n\bigr)}\..
\end{equation}
\end{thm}

\begin{proof}
By Theorem~\ref{t:VKLS}, we know that $\mu^{(n)}$ is
contained in $s\cdot \psi$ of area at most $m_n+dn^{2/3}$ for some
constant $d>0$.
Indeed, each of the $O(\sqrt{n})$ rows of $\mu^{(n)}$ differs from the
$\sqrt{n}\VK$ by at most $O(n^{1/6})$. Similarly, every
Plancherel $\wt\mu^{(n)}$ contains $t\cdot \psi$ of area at least
$m_n-dn^{2/3}$.

Let $\wt m_n = m_n + 2dn^{2/3}$.
By Theorem~\ref{t:LR-groups-exist}, for every Plancherel
partitions sequence there exists Plancherel sequences
$\{\wt\mu^{(n)}\vdash \wt m_n\}$ and $\{\nu^{(n)}\vdash n-\wt m_n\}$, s.t.
$$
c^{\la^{(n)}}_{\wt\mu^{(n)},\ts \nu^{(n)}} \, = \, \binom{n}{\wt m_n}^{1/2} \. e^{-O(\sqrt{n})}
\, = \, \binom{n}{m_n}^{1/2} \. e^{-O(n^{2/3}\log n)}.
$$
Note that
$$
f^{\nu^{(n)}} \, = \, \sqrt{(n-\wt m_n)!} \. e^{-O(\sqrt{n})}
\, = \, \sqrt{(n-m_n)!} \. e^{-O(n^{2/3}\log n)}.
$$
From above, $\mu_n \ssu t\cdot \psi \ssu \wt \mu_n$, where
$t=\sqrt{\theta n+ dn^{2/3}}$. Thus:
$$
f^{\la^{(n)}/\mu^{(n)}} \, \ge \, f^{\la^{(n)}/\wt\mu^{(n)}} \, \ge \,
c^{\la^{(n)}}_{\wt\mu^{(n)},\ts \nu^{(n)}} \cdot f^{\nu^{(n)}}
\, = \, \sqrt{\frac{n!}{m_n!}} \, e^{-O\bigl(n^{2/3}\ts\log n\bigr)}\..
$$
The upper bound follows from Corollary~\ref{c:SYT-asy}.  \end{proof}

\smallskip

We also have a partial converse:

\begin{thm}
Let \ts $m_n := \lfloor\theta n\rfloor$, where $\theta=1/2$.
Let $\{\la^{(n)}\vdash n\}$ and $\{\mu^{(n)}\vdash m_n\}$ be two
partitions sequences which satisfy
\begin{equation}
f^{\la^{(n)}/\mu^{(n)}} \, = \, \sqrt{\frac{n!}{(n/2)!}} \,  e^{O(n/\log n)}\..
\end{equation}
Then
$$
f^{\la^{(n)}} \, = \, \sqrt{(n/2)!} \,  e^{O(n)}\quad \ \text{and} \quad
f^{\mu^{(n)}} \, = \, \sqrt{(n/2)!} \,  e^{O(n)}\ts.
$$
 \end{thm}

\begin{proof}
Recall that
$$f^{\la/\mu} \, = \, \sum_{\nu\vdash n-m_n} \. c^\la_{\mu,\nu}\. f^\nu
$$
(see e.g.~\cite{EC2}).  Thus
$$c^\la_{\mu,\nu} \, \ge \, \frac{f^{\la/\mu}}{p(n) \ts \sqrt{|\la/\mu|!}}
$$
for some $\nu\vdash |\la/\mu|$.  Letting $\la \gets \la^{(n)}$, $\mu \gets \mu^{(n)}$
and $\nu \gets \nu^{(n)}$ corresponding to maximal LR--coefficient as above,
we conclude that
$$c^{\la^{(n)}}_{\mu^{(n)},\.\nu^{(n)}}
\, = \, \binom{n}{n/2}^{1/2}  \. e^{-O(n/\log n)}\..
$$
By Theorem~\ref{t:LR-intro2}, this implies the result.
\end{proof}

\medskip

\subsection{Proof of Theorem~\ref{t:LR-intro1}}\label{ss:skew-third-case}
Parts~$\WI$ and~$\WII$ are proved in~$\S$\ref{ss:LR-max}.  For~$\WIII$,
start with Theorem~\ref{t:skew-Plancherel} which gives
$$
f^{\la^{(n)}/\mu^{(n)}} \, = \, \sqrt{\frac{n!}{m_n!}} \,  e^{-O\bigl(n^{2/3}\ts\log n\bigr)}\..
$$
Write
$$f^{\la^{(n)}/\mu^{(n)}}  \, = \, \sum_{\nu\vdash n-m_n} \. c^{\la^{(n)}}_{\mu^{(n)},\.\nu}\.\ts f^\nu\.,
$$
and let $\{\nu^{(n)}\}$ be the largest terms in the summation.  Since
$$
c^{\la^{(n)}}_{\mu^{(n)},\.\nu^{(n)}} \, \le \, \sqrt{\binom{n}{m_n}} \qquad \text{and}
\qquad f^\nu \, \le \, \sqrt{(n-m_n)!}\.,
$$
we can follow the proof of Theorem~\ref{t:LR-groups-exist1} to obtain the result.  \ $\sq$

\bigskip

\section{Final remarks and open problems}\label{s:fin-rem}

\subsection{}\label{ss:finrem-hist}
There is a large body of work on Littlewood--Richardson coefficients and
its generalizations, but there are very few explicit formulas and general
inequalities. Even when such inequalities exist, they are difficult and
not very sharp see e.g.~\cite{BBR,CL,LPP}. The exact evaluations are even
more rare and use ad hoc methods, see e.g.~\cite{CDW}.  A notable sharp
lower bound in~\cite{Nar} applies only to partitions with the smallest
part $\la_\ell \ge \ell^2$.

For Kronecker coefficients, there are even fewer bounds, so much that
even positivity is known in very few cases, see~\cite{PP} for an overview.
We should note a strongly related work by Biane~\cite{Bia}, who
discusses the limit shape of the random Kronecker and LR--coefficients
weighted by the dimensions $f^\la$.

Finally, despite the \emph{Feit determinant formula} for
$f^{\la/\mu}$, getting good (theoretical) asymptotic estimates
for the number of standard Young tableaux of skew shape remains
difficult.  We refer to a survey article~\cite{AR} for exact formulas
(cf.~\cite{MPP3}), and to~\cite{MPP} for the upper and lower bounds.

\subsection{}\label{ss:finrem-groups}
Most results in the paper
on Kronecker coefficients extend directly to all finite groups
with few conjugacy classes.  This is a byproduct of the lack of
available tools for Kronecker coefficients.  On the other hand,
the LR--coefficients correspond to induced coefficients with
certain additional properties.  Notably, the results of
Section~\ref{s:LR-imply-dim} do not extend to general groups,
and dimensions of skew shape representations discussed in
Section~\ref{s:skew} cannot even be formulated.

\subsection{}\label{ss:finrem-mod}  A direct combinatorial proof of
Theorem~\ref{t:LR-modularity} would give a combinatorial interpretation
of the difference of LR--coefficients as in the theorem.  If such a
combinatorial interpretation is found, it would perhaps help establish
Conjecture~\ref{conj:LR-sub}.

\subsection{}\label{ss:finrem-LR}
Let us highlight one minor result:
\begin{equation}\label{eq:LR-minor}
\hskip.5cm c^\la_{\mu,\nu} \, \le \, \sqrt{\binom{n}{k}}\, \quad \ \ \text{for all} \quad
\la\vdash n, \, \mu \vdash k, \, \nu \vdash n-k
\end{equation}
(see Lemma~\ref{l:LR-squared-groups}).  This bound is very natural,
asymptotically tight (see Theorem~\ref{t:LR-intro-asy} below),
and has an elementary proof on the level of a graduate exercise.
Yet, to our astonishment, it seems to be new.

The inequality~\eqref{eq:LR-minor} should be compared to
$$
c^\la_{\mu,\nu} \, \le \, f^\nu \quad \text{and} \quad c^\la_{\mu,\nu} \, \le \, \frac{f^{\la/\mu}}{f^\nu}\.,
$$
(see~$\S$\ref{ss:LR-mono}), which are cleaner but asymptotically much weaker.

Note also that our proof of~\eqref{eq:LR-minor} can be phrased as a direct
combinatorial argument, using
a double counting surjection.  Formally, the inequality
$$
f^\la \ts f^\mu \ts f^\nu \. (c^\la_{\mu,\nu})^2 \, \le f^\la \ts f^\mu \ts f^\nu \. \binom{n}{k}
$$
is a
combination of two surjections, both of which can be obtained by a jeu-de-taquin argument
using two different combinatorial interpretation of $c^\la_{\mu,\nu}$ in terms of
Young tableaux (cf.~\cite{Ker,vL,Zel}).  We leave the details to the reader.

\subsection{}\label{ss:finrem-table}  For small values of~$n$ and~$k$,
the numerical values of the lower and upper bounds given by
Proposition~\ref{p:LR-groups-bounds} are too far apart to be useful.
For example, $\rC(20,7)=11$ (see~$\S$\ref{a:LR}), while
the proposition gives
$$
0.28\, \approx \,  \frac{1}{\sqrt{p(7)\, p(13) \, p(20)}} \. \sqrt{\binom{20}{7}}\,\.
\, \le \, \rC(20,7)  \, \le \, \sqrt{\binom{20}{7}}
\, \approx \, 278.42\ts.
$$
Even the upper bound \ts $\rC(20,7) \le \rD(7) = 35\le \sqrt{7!}\approx 70.99$ \ts in Theorem~\ref{t:LR-stab}
is better in this case.

\subsection{}\label{ss:finrem-conj0}
By analogy with the upper bound in~\eqref{eq:VK}, we believe that $\rC(n)$ are rather
small when compared to $2^{n/2}$. We conjecture that
$$\rC(n) \, \le \, 2^{n/2} \. e^{-a\sqrt{n}} \quad \text{for some} \ \, a>0.
$$
This suggests that $\rC(n)$ is somewhere in between the lower and upper bound in
Proposition~\ref{p:LR-groups-bounds}.  Note that Conjecture~\ref{conj:LR-upper-dim} 
comes close: it implies the upper bound \ts $2^{n/2} \. e^{-a\sqrt{n}/\log n}$, 
which would also be interesting.  

\subsection{}\label{ss:finrem-conj1}
We conjecture that Theorem~\ref{t:LR-intro2} can be extended to
general $0<\theta <1$. For that, one can try to improve the lower bound in
Theorem~\ref{t:LR-upper-dim} (see also Remark~\ref{rmk:ineq}).
This problem has an especially elegant special case in terms of Biane's
approach~\cite{Bia}, which can be presented as follows.

Let $\{\mu^{(n)} \vdash n/2\}$ and $\{\nu^{(n)} \vdash n/2\}$ be partition
sequences with strongly stable shapes $\om$ and~$\pi$, respectively.
Suppose \ts $\area(\om)=\area(\pi)$.
Let \ts $\phi = \om \boxplus \pi$ \ts be the \emph{free convolution}
of the limit shape functions (see \cite{Bia} for definitions),
and let $\{\la^{(n)} \vdash n\}$ be partitions
sequence with strongly stable shape~$\phi$. It is easy to see that
Stanley's upper bound~$(\ast\ast)$ in Theorem~\ref{t:stanley-asy}
in terms of hook integrals (see~$\S$\ref{ss:def-limit}),
combined with Biane's concentration results gives
\begin{equation}\label{eqhook-int-ineq}
\Ups(\phi) \. \ge \. \frac{\Ups(\om) + \Ups(\pi)}{2}\.,
\end{equation}
with the equality for VKLS shapes.

\begin{conj}
Equality in~\eqref{eqhook-int-ineq} holds only for VKLS shapes.
\end{conj}

\subsection{}\label{ss:finrem-conj2}  It would be interesting to
sharpen our bounds in Theorem~\ref{t:LR-refined-row} for LR--coefficients
with few rows.  Perhaps, one can then ask for about the shape of the largest
such LR--coefficients.  Although our proof is combinatorial, our bounds are
too far apart to be useful.  Note also that this is analogous to the
$1$-parametric family of limit shapes with largest dimension
and bound on the number of rows obtained by Logan and Shepp~\cite{LS}.

\subsection{}\label{ss:finrem-conj3}  As we mentioned earlier there are
no explicitly constructed series of partitions for which Kronecker
coefficients are $\exp\Omega(n)$.  In view of Theorem~\ref{t:kron-intro},
it would be especially interesting to obtain an explicit construction
giving a $\exp \Omega(n\log n)$.  We conjecture that the staircase shapes
in the Saxl conjecture give such examples, but as we mentioned
in~$\S$\ref{ss:Kron-saxl}, we are nowhere close to resolving this problem.

\vskip.76cm

\subsection*{Acknowledgements}
We are grateful to Richard Stanley for introducing us
to the area; his questions and results were the inspiration
for this whole paper.  We would like to thank Sami Assaf, Vadim Gorin,
Alejandro Morales, Leonid Petrov, Dan Romik, Dima Schlyaktenko
and Martin Tassy for helpful comments and insightful remarks.
The first and second authors were partially supported by the~NSF.

\newpage


\newpage
\appendix

\section{Table of the largest Littlewood--Richardson coefficients}\label{a:LR}

\
\vskip.3cm

\nin
Here is the sequence $\{\rC(n)\}$.

\vskip.3cm

\hskip-.7cm
\begin{tabular}{|c|ccccccccccccccccccccccc|}
 \hline
\text{$n$} & 1 & 2 & 3 & 4 & 5 &  6 & 7 & 8 &  9 & 10  & 11 & 12  & 13 & 14 & 15 & 16 & 17 & 18 &  19 & 20 & 21& 22 & 23 \\
\hline
\text{$\rC(n)$}  &  1 & 1 &  1 & 1  & 1   & 2  & 2  &  2 &  2 & 3
 & 3  &  4  & 4   & 5   &  6  & 8   &  9  &  11  &  12  &  18  &  24  &  32 & 35\\
 \hline
\end{tabular}

\vskip.95cm

\nin
Here is the table of $\rC(n,k)$, for \ts $0 \le k \le 20$ \ts
and \ts $1 \le n \le 23$. \ts Note that \ts $\rC(n,k) = \rC(n,n-k)$.

\vskip.3cm

\hskip-.7cm
\begin{tabular}{|c|rrrrrrrrrrrrrrrrrrrrr|}
 \hline

\text{$n \setminus k$} & \, 0& \, 1 & \, 2 & \, 3 & \, 4 & \, 5 &  6 & 7 & 8 &  9 &
10  & 11 & 12  & 13 & 14 & 15 & 16 & \!\!\! 17 & \!\!\! 18 & \!\!\! 19 & \!\!\! 20 \\
\hline
1  &  1 & 1 &   &   &    &   &   &   &   &   &   &    &    &    &    &    &    &    &    &    &    \\
2  &  1 & 1 & 1 &   &    &   &   &   &   &   &   &    &    &    &    &    &    &    &    &    &    \\
3  &  1 & 1 & 1 & 1 &    &   &   &   &   &   &   &    &    &    &    &    &    &    &    &    &    \\
4  &  1 & 1  & 1 & 1 & 1  &   &   &   &   &   &   &    &    &    &    &    &    &    &    &    &   \\
5  &  1 & 1  & 1 & 1 & 1  & 1  &   &   &   &   &   &    &    &    &    &    &    &    &    &    &  \\
6  &  1 & 1  & 1& 2 &  1  & 1  & 1  &   &   &   &   &    &    &    &    &    &    &    &    &    & \\
7  &  1 & 1  & 1 & 2 & 2   & 1  & 1  & 1  &   &   &   &    &    &    &    &    &    &    &    &     & \\
8  &  1 & 1  & 1 & 2 &  2  & 2  &1   &  1 & 1  &   &   &    &    &    &    &    &    &    &    &    & \\
9  &  1 & 1  & 1 & 2 &  2  & 2  & 2  & 1  & 1  & 1  &   &    &    &    &    &    &    &    &    &    & \\
10  &  1 & 1 & 1 & 2 &  3 & 2  & 3  & 3  &  1 & 1  & 1  &    &    &    &    &    &    &    &    &   & \\
11  &  1 & 1 & 1 & 2 & 3  & 3  & 3  & 3  & 2  & 1  & 1  & 1   &    &    &    &    &    &    &    &  & \\
12 &  1 & 1  & 1 & 2 & 3  & 3  & 4  & 3  & 3  & 2  & 1  &  1  & 1   &    &    &    &    &    &    &  & \\
13 &  1 & 1  & 1 & 2 & 3  & 3  & 4  & 4  & 3  & 3  & 2  & 1   & 1   & 1   &    &    &    &    &     &  & \\
14 &  1 & 1  & 1 & 2 & 3  & 3  & 4  & 5  & 4  & 3  & 3  & 3   & 1   & 1   & 1   &    &    &    &      &  & \\
15 &  1 & 1  & 1 & 2 & 3  & 6  & 6  & 5  & 5  & 6  & 6  & 3   & 2   & 1   & 1   &  1  &    &    &     &  & \\
16 &  1 & 1  & 1 & 2& 3 & 6 & 8  & 7  & 6  & 7  & 8  & 6  & 3   & 2   & 1   & 1   & 1   &    &    &   &   \\
17 &  1 & 1  & 1 & 2& 3 & 6 & 8  & 9  & 8 & 8  & 9  & 8  & 6  & 3   & 2   & 1   & 1   & 1   &    &   &\\
18 &  1 & 1  & 1 & 2& 3 & 6 & 8  & 11  & 10 & 9  & 10  & 11  & 8  &  6  & 3   & 2   & 1   & 1   & 1  & &  \\
19 &  1 & 1  & 1 & 2& 3 & 6 & 8  & 11  & 12 & 11  & 11 & 12  & 11 & 8  &  6  & 3   & 2   & 1   & 1   & 1   &  \\
20 &  1 & 1  & 1 & 2& 3 & 6 & 8  & 11  & 12 & 13  & 18 & 13 & 12  & 11 & 8  &  6  & 3   & 2   & 1   & 1   & 1  \\
21 &  1 & 1  & 1 & 2& 3 & 6 & 16 & 12 & 14  & 14 & 24 & 24  & 14 & 14 & 12 & 16  &  6  & 3   & 2   & 1   & 1   \\
22 &  1 & 1  & 1 & 2& 3 & 6 & 16 & 20 & 15  & 16 & 27 & 32 & 27  & 16 & 15 & 20 & 16  &  6  & 3   & 2   & 1    \\
23 &  1 & 1  & 1 & 2& 3 & 6 & 16 & 20 & 24  & 19 & 30 & 35 & 35 & 30  & 19 & 24 & 20 & 16  &  6  & 3   & 2     \\
 \hline
\end{tabular}

\vskip.95cm

\nin
For comparison, here is the sequence $\{\rD(n)\}$ taken from~\cite[\href{https://oeis.org/A003040}{A003040}]{OEIS},
to which the column values in the table above stabilize by Theorem~\ref{t:LR-stab}.

\vskip.3cm

\hskip-.7cm
\begin{tabular}{|c|cccccccccccccccc|}
 \hline
\text{$n$} & \. 1 & \. 2 & \. 3 & \. 4 \. & 5 &  6 & 7 & 8 &  9 & 10  & 11 & 12  & 13 & 14 & 15 & 16   \\
\hline
\text{$\rD(n)$}  &  \. 1 & \. 1 & \. 2 & \ts 3  & \ts 6   & 16  & 35  &  90 &  216 & 768
 & 2310  &  7700  & 21450   & 69498   &  292864  & 1153152    \\
 \hline
\end{tabular}

\end{document}